\documentclass[]{siamart}

\usepackage{lipsum}
\usepackage{amsfonts}
\usepackage{graphicx,epstopdf}
\usepackage[caption=false]{subfig}
\usepackage{epstopdf}
\usepackage{algorithmic}
\usepackage{enumitem}
\usepackage[sort,nocompress]{cite}
\ifpdf
  \DeclareGraphicsExtensions{.eps,.pdf,.png,.jpg}
\else
  \DeclareGraphicsExtensions{.eps}
\fi

\newcommand{\TheTitle}{Conservative methods for dynamical systems} 
\newcommand{\TheAuthors}{Andy T. S. Wan, Alexander Bihlo and Jean-Christophe Nave}

\headers{\TheTitle}{\TheAuthors}

\title{{\TheTitle}\thanks{ATSW was supported by the Centre de recherches math\'ematiques. AB was supported by the Canada Research Chairs program and NSERC Discovery Grant program. JCN was supported by the NSERC Discovery Grant progam and NSERC Discovery Accelerator Supplement programs.}}

\author{
  Andy T. S. Wan\thanks{Department of Mathematics and Statistics, McGill University, Montr\'eal, QC, H3A 0B9, Canada
    (\email{andy.wan@mcgill.ca}).}
    \and
    Alexander Bihlo\thanks{Department of Mathematics and Statistics, Memorial University of Newfoundland, St.\ John's, NL, A1C 5S7, Canada
    (\email{abihlo@mun.ca}).}
    \and
  Jean-Christophe Nave\thanks{Department of Mathematics and Statistics, McGill University, Montr\'eal,  QC, H3A 0B9, Canada (\email{jcnave@math.mcgill.ca}).}
}

\usepackage{amsopn}

\usepackage{amstext,amsmath,amssymb}
\newtheorem{example}{Example}
\newsiamthm{claim}{Claim}
\newsiamthm{remark}{Remark}
\newsiamthm{hypo}{Hypothesis}

\def\[#1\]{\begin{align*}#1\end{align*}}

\newcommand\norm[1]{\left\lVert#1\right\rVert}

\newcommand{\bb}{\boldsymbol }

\ifpdf
\hypersetup{
  pdftitle={\TheTitle},
  pdfauthor={\TheAuthors}
}
\fi

\begin{document}

\maketitle

\begin{abstract}
We show a novel systematic way to construct conservative finite difference schemes for quasilinear first-order system of ordinary differential equations with conserved quantities. In particular, this includes both autonomous and non-autonomous dynamical systems with conserved quantities of arbitrary forms, such as time-dependent conserved quantities. Sufficient conditions to construct conservative schemes of arbitrary order are derived using the multiplier method. General formulas for first-order conservative schemes are constructed using divided difference calculus. New conservative schemes are found for various dynamical systems such as Euler's equation of rigid body rotation, Lotka--Volterra systems, the planar restricted three-body problem and the damped harmonic oscillator.
\end{abstract}




\section{Introduction}\text{ }\\
Conservative methods for ordinary differential equations (ODEs) are numerical methods which preserve their first integrals, invariants or, equivalently, conserved quantities. One primary motivation behind their use is the intrinsic long-term stability properties which conservative methods can possess, making them important in the long time study of dynamical systems in fields such as, but not limited to, astronomy, molecular dynamics, fluid mechanics, climate prediction and mathematical biology.

In the past decades, considerable research efforts have been devoted to developing methods which preserve both approximately or exactly (up to round-off errors) first integrals \cite{LaBGre75, Qui97, MclQuiRob98, McL99, MclQui04, DahOwrYag11, cell12a, NorQui14, NorAE15}. For Hamiltonian systems, symplectic methods and variational integrators approximately conserve the energy and also preserve other important underlying geometric structures. Specifically, symplectic methods preserve the symplectic structure of canonical Hamiltonian systems so that their discrete flow is volume-preserving. Moreover, their energy is nearly conserved over an exponentially long time \cite{CalHai95}. Variational integrators preserve Hamilton's principle of stationary action at a discrete level and as a result, they are symplectic and can exactly conserve momenta arising from symmetries due to a discrete version of Noether's theorem \cite{MarWes01}. 

Other approximate and exactly conservative methods exist for special classes of ODEs or special forms of invariants. For instance, Kahan's method applies to at most quadratic vector fields and exactly preserves a modified Hamiltonian when applied to Hamiltonian systems \cite{Cel09}. Although all Runge-Kutta methods exactly preserve linear invariants \cite{Sha86} and, in some cases, quadratic invariants \cite{coo87}, it is known that no consistent Runge-Kutta method can exactly preserve arbitrary polynomial invariants with degree three or higher \cite{Cal97}. However, in the case of Hamiltonian systems, specific Runge-Kutta methods can be constructed to exactly preserve the energy of a particular polynomial form \cite{Cel09}. Moreover, the average vector field method exactly preserves general energy functions using an integral formula \cite{QuiMcL08}. 

In the general case, to the best of our knowledge, there are two known classes of exactly conservative methods for ODEs, specifically for quasilinear systems\footnote{Here, quasilinear means the ODE system is linear in its highest time derivative.}. One class of methods is called the discrete gradient method \cite{Qui97, McL99} which can exactly preserve arbitrary forms of first integrals. The general idea is to rewrite the ODE system in a skew-gradient form so that a discrete gradient method can be used to exactly preserve their first integrals. Difficulties can arise in the reformulation at degenerate critical points of the first integrals and in constructing sparse skew-symmetric tensors for large systems. Moreover, the discrete gradient method has so far only been developed for time-independent conserved quantities.

The other general class of exactly conservative method are projection methods. The main idea is to advance in time using any numerical scheme and project its solution back onto the level set of the invariants after some number of time steps \cite{hair06Ay}. The projection step is generally performed by solving a constrained optimization problem using Lagrange multipliers, which can be computationally expensive for large systems. Thus, a projection step is usually taken only once every few time steps. While projection methods are exactly conservative in general, this approach may not possess long-term stability properties as discussed in \cite{wan16}. Specifically, the projection step may project onto a different connected component of the level set of the invariants leading to instability over long time.

In contrast to the well-known quantities such as energy and momentum, there can be conserved quantities for dynamical systems which may not have an a priori physical meaning. One example is the time-dependent conserved quantity for the damped harmonic oscillator observed in \cite{wan15a}. To treat arbitrary forms of conserved quantities of ODEs and as well as conservation laws of partial differential equations (PDEs), the multiplier method was introduced in \cite{wan15a} as a general conservative method. The main idea is to discretize the so-called characteristics \cite{olve86Ay} or conservation law multipliers \cite{blum10Ay} of a PDE system so that a discrete divergence theorem holds. It is known from \cite{olve86Ay} that for a normal, nondegenerate PDE system, there is a one-to-one correspondence between equivalence classes of conservation laws and their associated multipliers. Thus, the multiplier method can be applied to virtually all differential equations arising in practical applications without any additional geometric structures, such as non-Hamiltonian systems or nonautonomous systems. In particular, for quasilinear first-order systems, we will show that the multiplier method can exactly preserve arbitrary forms of conserved quantities, such as time-dependent conserved quantities, and can be applied directly without any reformulation or transformation of the given ODE system. This is especially important for constructing conservative semi-discretizations of PDEs, where arbitrary large systems can arise which necessitates a systematic construction of conservative schemes without additional reformulation or transformation.

In this paper, we specialize the multiplier method to quasilinear first-order systems and show that conservative schemes can be systematically constructed for general forms of conserved quantities. This paper is organized as follows. In Section~\ref{sec:notation}, we introduce notations and conventions used throughout the paper. In Section~\ref{sec:multTheory}, we give a short introduction on the theory of conservation law multipliers for quasilinear ODEs. It is shown that conserved quantities and conservation law multipliers share two important relations. In Section~\ref{sec:multMethod}, these relations are discretized leading to
sufficient conditions for constructing conservative schemes of arbitrary order. General formulas for conservative schemes of (at least) first order are provided using divided difference calculus. In Section~\ref{sec:CM_examples}, conservative schemes are derived systematically for the Euler's equation of rigid body rotation, Lotka--Volterra systems, the planar restricted three-body problem and the damped harmonic oscillator. In Section~\ref{sec:numerics}, exact conservation properties are numerically verified for the derived examples up to round-off errors. And finally, some concluding remarks are given on future work in the conclusion.

\section{Notations and conventions}\label{sec:notation}\text{ }\\
Let $U\subset \mathbb{R}^n$ and $V\subset \mathbb{R}^m$ be open subsets where here and in the following $n,m,p \in \mathbb{N}$. $f\in C^p(U\rightarrow V)$ means $f$ is a $p$-times continuously differentiable function with domain in $U$ and range in $V$. When the domain and range of $f$ is clear, we simply write $f\in C^p$. If $p=0$, we write $f\in C(U\rightarrow V)$ to indicate that $f$ is a continuous function on $U$. We often use boldface to indicate a vectorial quantity $\bb f$. If $\bb f\in C^p(V\rightarrow \mathbb{R}^n)$, $D^p\bb f$ denotes the $p$ partial derivatives of $\bb f$ with respect to $\bb x$ and $\partial_{\bb x} \bb f:=\left[\frac{\partial f_i}{\partial x_j}\right]$ denotes the Jacobian matrix. Let $I\subset \mathbb{R}$ be an open interval and let $\bb x\in C^1(I\rightarrow U)$, $\dot{\bb x}$ denotes the derivative with respect to time $t\in I$. Also if $\bb x\in C^p(I\rightarrow U)$, $\bb x^{(q)}$ denotes the $q$-th time derivative of $\bb x$ for $1\leq q\leq p$. For the sake of brevity, the explicit dependence of $\bb x$ on $t$ is often omitted with the understanding that $\bb x$ is to be evaluated at $t$. If $\bb \psi\in C^1(I\times U\rightarrow \mathbb{R}^m)$, $D_t \bb\psi$ denotes the total derivative with respect to $t$, and $\partial_t \bb\psi$ denotes the partial derivative with respect to $t$. $M_{m\times n}(\mathbb{R})$ denotes the set of all $m\times n$ matrices with real entries. For a vector $\bb v$ or matrix $A$, $\bb v^T$ and $A^T$ denote their transposes. The usual dot product of two vectors $\bb v, \bb w \in \mathbb{R}^n$ is denoted by $\bb v\cdot \bb w$ or equivalently by $\bb v^T\bb w$. $\bb e_i\in \mathbb{R}^n$ denotes the $i$-th standard basis vector.

\section{Theory of conservation law multipliers for quasilinear ODEs}
\label{sec:multTheory}
\text{ }\\
In this section, we present a short self-contained description on the theory of conservation law multipliers for quasilinear first-order systems of ODEs\footnote{As usual, quasilinear ODEs with higher order derivatives can be transformed into quasilinear first-order systems of ODEs by introducing auxiliary variables for the higher order derivatives.} with real entries and $C^p$ solutions. In particular, this includes both autonomous or non-autonomous dynamical systems.\footnote{Indeed, a non-autonomous system can always be transformed into an autonomous one. However, we will show that the proposed conservative methods can be applied to the system ``as is" and thus they are directly applicable to semi-discretizations of PDEs.} In essence for quasilinear first-order ODEs, we show that there is a one-to-one correspondence between its conserved quantities and so-called \emph{conservation law multipliers}. 
Although this correspondence principle is not new, as \cite{olve86Ay} proved a more general correspondence for equivalent classes of conservation laws for \emph{normal, nondegenerate} PDE systems with locally $C^\infty$ functions. The new proofs presented here have the advantage of being quite elementary and thus are more amendable to non-specialists and for extensions to weaker regularity assumptions. 

Consider a quasilinear first-order system of ordinary differential equations,
\begin{align}
\bb F(t,\bb x, \dot{\bb x}):=\dot{\bb x}(t) - \bb f(t,\bb x) &= \bb 0, \label{odeEqn} \\
\bb x(t_0)&=\bb x_0. \nonumber
\end{align} where $t\in I$, $\bb x=(x_1(t),\dots, x_n(t)) \in U$. 
For $p\in \mathbb{N}$, we consider the case $\bb f \in C^{p-1}(I\times U\rightarrow\mathbb{R}^n)$. Thus, by standard ODE theory, there exists an unique solution $\bb x\in C^p(I\rightarrow U)$ to the first order system \eqref{odeEqn} in a neighborhood of $(t_0,\bb x_0)\in I\times U$. We shall assume $I$ is such maximal interval of existence.

\subsection{Conserved quantities of quasilinear first order ODEs}
\begin{definition} Let $m,q\in\mathbb{N}$ with $1\leq m$ and $q\leq p-1$. For $i=1,\dots, q$, let $U^{(i)}$ be open subsets of $\mathbb{R}^n$. A $q$-th order conserved quantity of $\bb F$ is a vector-valued function $\bb \psi \in C^1(I\times U\times U^{(1)}\times\cdots\times U^{(q)}\rightarrow \mathbb{R}^m)$ such that
\begin{align}
D_t \bb \psi(t,\bb x, \dot{\bb x},\dots, \bb x^{(q)}) = \bb 0, \text{ for any $t\in I$ and $C^p$ solution } \bb x \text{ of }\bb F. \label{CLdef}
\end{align}
\end{definition}
It follows that $\bb \psi(t,\bb x, \dot{\bb x},\dots, \bb x^{(q)})$ is constant on $t\in I$ for any $C^p$ solution $\bb x$ of $\bb F$. By the quasilinear property of $\bb F$, we can differentiate both sides of \eqref{odeEqn} with respect to $t$ up to $q-1$ times and substitute time derivatives of $\bb x$ with partial derivatives of $\bb f$:
\[
\bb x^{(2)} &= D_t \bb f(t,\bb x) = \bb f_1(t,\bb x,\bb f(t, \bb x), D\bb f(t, \bb x) )=:\tilde{\bb f}_1(t,\bb x), \\
&\text{ }\vdots \\
\bb x^{(q)} &= D_t^{(q-1)} \bb f(t,\bb x) = \bb f_{q-1}(t,\bb x,\bb f(t, \bb x), D\bb f(t, \bb x),\dots, D^{(q-1)}\bb f(t, \bb x))=:\tilde{\bb f}_q(t,\bb x).
\]
Thus, we can rewrite a $q$-th order conserved quantity $\bb \psi$ as a zeroth order conserved quantity $\tilde{\bb \psi}(t,\bb x):=\bb \psi(t, \bb x, \tilde{\bb f}_1(t, \bb x),\dots, \tilde{\bb f}_{q-1}(t, \bb x))$ so that $D_t\tilde{\bb \psi}(t,\bb x) = \bb 0$ for any $C^p$ solution $\bb x$ of $\bb F$. In other words, without loss of generality, it suffices to consider only zeroth order conserved quantities $\bb \psi(t,\bb x)$ for $\bb F$ in \eqref{CLdef}.

\subsection{Conservation law multipliers for quasilinear first order ODEs}
\text{ }\\
Next, we introduce a generalization of integrating factors referred to as \emph{characteristics} by \cite{olve86Ay} or equivalently, \emph{conservation law multipliers} by \cite{blum10Ay}. We will adopt the terminology of conversation law multiplier or just multiplier when the context is clear.\vskip -2mm
\begin{definition}
A conservation law multiplier of $\bb F$ is a matrix-valued function $\Lambda \in C(I\times U \times U^{(1)}\rightarrow M_{m\times n}(\mathbb{R}))$ such that there exists a function $\bb \psi \in C^1(I\times U\rightarrow \mathbb{R})$,
\begin{align}
\Lambda(t,\bb x, \dot{\bb x})(\dot{\bb x}(t)-\bb f(t,\bb x)) = D_t \bb \psi(t,\bb x), \text{ for $t\in I$, }\bb x \in C^1(I\rightarrow U). \label{multEqn}
\end{align}
\end{definition}
Here, we emphasize that condition \eqref{multEqn} holds as an identity for arbitrary $C^1$ functions $\bb x$; that is, $\bb x$ does not need to be a solution of $\bb F$. Moreover, in general, there can be many conservation law multipliers for the same function $\bb \psi$. However, if we restrict to multipliers of the form $\Lambda(t,\bb x)$, then there is a one-to-one correspondence between conservation law multipliers of $\bb F$ and conserved quantities of $\bb F$, up to constant factors in $\bb \psi$. To show this, we need the following lemma.

\begin{lemma}
Let $\bb g\in C(I\times U\rightarrow \mathbb{R}^m)$ and suppose $\bb g$ has the property that $\bb g(t,\bb x)=\bb 0$ for any $C^1$ solution $\bb x$ of $\bb F$, then $\bb g(t,\bb x)=0$ for all $(t, \bb x)\in I\times U$. In other words, $\bb g$ is identically the zero function on $I\times U$.  \label{vanishLem}
\end{lemma}
\begin{proof}
Suppose $\bb g(t_0,\bb x_0)\neq \bb 0$ for some $t_0\in I, \bb x_0\in U$. Local existence of solution to $\bb F$ implies there exists a $C^1$ (in fact $C^p$) solution $\bb x$ with $\bb x(t_0)=\bb x_0$. However, $\bb 0\neq \bb g(t_0,\bb x(t_0))$ contradicts the hypothesis that $\bb g$ vanishes on any $C^1$ solution of $\bb F$.
\end{proof}

Now, we are in the position to show the key correspondence result for conservation law multipliers of the form $\Lambda(t,\bb x)$.

\begin{theorem}[Correspondence Theorem] \label{corrThm}
Let $\bb \psi\in C^1(I\times U\rightarrow \mathbb{R}^m)$. Then there exists a unique conservation law multiplier of $\bb F$ of the form $\Lambda \in C(I\times U\rightarrow M_{m\times n}(\mathbb{R}))$ associated with the function $\bb \psi$ if and only if $\bb \psi$ is a conserved quantity of $\bb F$. And if so, $\Lambda$ is unique and satisfies for any $t\in I$ and $ \bb x\in C^1(I\rightarrow U)$,
\begin{subequations}
\begin{align}
\Lambda(t,\bb x) = \partial_{\bb x} \bb \psi(t,\bb x), \label{multCond1} \\
\Lambda(t,\bb x) \bb f(t,\bb x) = -\partial_t \bb \psi(t,\bb x). \label{multCond2}
\end{align}
\end{subequations}
\end{theorem}
\begin{proof}Suppose $\Lambda \in C(I\times U\rightarrow M_{m\times n}(\mathbb{R}))$ is a conservation law multiplier of $\bb F$ associated with the function $\bb \psi$. It follows immediately from the definition of conservation law multipliers that for any $\bb x\in C^1$ solution of $\bb F$, 
\[
D_t\bb \psi(t, \bb x) = \Lambda(t,\bb x) (\dot{\bb x}(t)-f(t,\bb x)) = 0.
\]
Thus, the function $\bb \psi$ is a conserved quantity of $\bb F$.
To show that $\Lambda$ satisfy \eqref{multCond1} and \eqref{multCond2}, note that for any $\bb x\in C^1(I\rightarrow U)$, 
\begin{align}
\Lambda(t,\bb x)(\dot{\bb x}(t)-\bb f(t,\bb x)) = D_t\bb \psi(t, \bb x) &= \partial_{\bb x} \bb \psi(t,\bb x)\cdot \dot{\bb x}(t)+\partial_{t} \bb \psi(t,\bb x) \label{multRelation1}
\end{align}
Choosing the constant function $\bb x(t)=\bb y\in U$ in \eqref{multRelation1} implies $\Lambda$ satisfies \eqref{multCond2} for any $(t, \bb y)\in I\times U$. Since \eqref{multCond2} is now satisfied, \eqref{multRelation1} simplifies to
\begin{align}
\Lambda(t,\bb x)\dot{\bb x}(t) = \partial_{\bb x} \bb \psi(t,\bb x)\cdot \dot{\bb x}(t), \text{ for any }\bb x\in C^1(I\rightarrow U). \label{multRelation2}
\end{align} For $i=1,\dots, n$ and any $s\in I, \bb y\in U$, choosing the linear function $\bb x(t)=\bb e_i(t-s)+\bb y$ and evaluating \eqref{multRelation2} at $t=s$ shows the $i$-th column of $\Lambda(s,\bb y)$ and $\partial_{\bb x} \bb \psi(s,\bb y)$ are equal, which implies $\Lambda$ also satisfies \eqref{multCond1}.

Conversely, let $\bb \psi(t,\bb x)$ be a conserved quantity of $\bb F$. Define $\bb\Lambda (t, \bb x) = \partial_{\bb x} \bb \psi(t,\bb x)$ as given in \eqref{multCond1}. Then, for any $\bb x\in C^1(I\rightarrow U)$,
\[
D_t\bb \psi(t, \bb x) &= \partial_{\bb x} \bb \psi(t,\bb x)\cdot \dot{\bb x}(t)+\partial_t\bb \psi(t,\bb x) \\
&= \partial_{\bb x} \bb \psi(t,\bb x)\cdot (\dot{\bb x}(t)-\bb f(t,\bb x)) + \underbrace{\Lambda(t,\bb x)\bb f(t,\bb x) +\partial_t \bb \psi(t,\bb x)}_{=:\bb g(t,\bb x)}
\] Since $\bb \psi$ is a conserved quantity of $\bb F$, for any $C^1$ solution $\bb x$ of $\bb F$,
\[
\bb g(t,\bb x) = \underbrace{D_t\bb \psi(t, \bb x)}_{=0} - \partial_{\bb x} \bb \psi(t,\bb x)\cdot \underbrace{(\dot{\bb x}(t)-\bb f(t,\bb x))}_{=0} = 0.
\] In other words, $\bb g(t,\bb x)=0$ on any $C^1$ solution $\bb x$ of $\bb F$. Thus, by Lemma \ref{vanishLem}, $\bb g(t,\bb x)$ is identically zero which implies \eqref{multCond2}.
\end{proof}
\begin{remark}
It is possible for an autonomous system $\bb F$ to have time dependent conserved quantities, as Example \ref{dhoSys} of Section \ref{exCLMult} will illustrate.
\end{remark}
\begin{remark}
Note that on the solutions of $\bb F$, a conserved quantity $\bb \psi$ of $\bb F$ is advected by the flow velocity $\bb f$ given by, 
\[ \partial_t \bb \psi(t,\bb x) + \partial_{\bb x}\bb \psi(t,\bb x) \bb f(t,\bb x) = 0.\] Indeed, this is implied by conditions \ref{multCond1}--\ref{multCond2} restricted to $C^1$ solutions of $\bb F$. Moreover, in fact, conditions \ref{multCond1}--\ref{multCond2} say a conserved quantity $\bb \psi$ of $\bb F$ satisfies the above advection equation for any $\bb x\in C^1(I\rightarrow U)$.
\end{remark}

Theorem \ref{corrThm} is useful in constructing conserved quantities of $\bb F$. In particular, it is enough to consider multipliers of the form $\Lambda(t,\bb x)$, which will make computation much simpler in practice. We will make use of conditions \eqref{multCond1} and \eqref{multCond2} to systematically construct conservative discretizations for $\bb F$ in Section \ref{sec:multMethod}. 

For a known conserved quantity $\bb \psi(t,\bb x)$ of $\bb F$, the corresponding conservation law multiplier $\Lambda(t,\bb x)$ can be computed using \eqref{multCond1}. Moreover in general, even if a conserved quantity of $\bb F$ is not known in advance, conserved quantities may be found using the Euler operator \cite{olve86Ay}. For brevity, here we only define the Euler operator for functions of the form $\bb g(t, \bb x, \dot{\bb x})$, though similar results hold for higher order partial derivatives. \vskip -3mm
\begin{definition} Let $\bb g \in C^1(I\times U\times U^{(1)}\rightarrow \mathbb{R}^m)$. For $i=1,\dots, n$, the Euler operator of $\bb g$ is the linear operator $E:C^1(I\times U\times U^{(1)}\rightarrow \mathbb{R}^m)\rightarrow \mathbb{R}^m$ defined by,
\[
(E\bb g)(t,\bb x,\dot{\bb x}) := \partial_{\bb x} \bb g(t, \bb x, \dot{\bb x}) - \left(D_t\circ \partial_{\dot{\bb x}}\right) \bb g(t, \bb x, \dot{\bb x}), \text{ for any } \bb x\in C^{1}(I\rightarrow U).
\]
\end{definition}

\begin{theorem}[Euler operator] \label{EulerOp} Let $I\times U\times U^{(1)}$ be a star-shaped domain centered at $(t_0,\bb x_0, \dot{\bb x}_0)$ and $\bb g\in C^1(I\times U\times U^{(1)}\rightarrow \mathbb{R}^m)$. Then $(E \bb g)(t,\bb x, \dot{\bb x}) = \bb 0$ for all $(t,\bb x, \dot{\bb x})\in I\times U\times U^{(1)}$ if and only if there exists $\bb \psi\in C^2(I\times U\times U^{(1)}\rightarrow \mathbb{R}^m)$ such that $\bb g(t,\bb x, \dot{\bb x})=D_t \bb \psi(t,\bb x, \dot{\bb x})$ for all $(t,\bb x, \dot{\bb x})\in I\times U\times U^{(1)}$. 
\end{theorem}
\begin{proof}
The proof follows analogously from the case when $I=\mathbb{R}, U=\mathbb{R}^n = U^{(1)}$ proved in Theorem 4.7 of \cite{olve86Ay} with $(t_0,\bb x_0, \dot{\bb x}_0)$ translated to the origin. The star-shaped domain condition is only needed in the forward implication where a line integration is used to construct $\bb \psi$.
\end{proof} Since conservation law multipliers satisfy \eqref{multEqn}, Theorem \ref{EulerOp} implies the following:
\begin{corollary} \label{EulerCond}
Let $I\times U\times U^{(1)}$ be a star-shaped domain centered at $(t_0,\bb x_0, \dot{\bb x_0})$. Then $\Lambda \in C^1(I\times U\rightarrow M_{m\times n}(\mathbb{R}))$ is a conservation law multiplier of $\bb F$ if and only if $\left(E (\Lambda \bb F)\right)(t,\bb x, \dot{\bb x}) = \bb 0$ for all $(t,\bb x, \dot{\bb x})\in I\times U\times U^{(1)}$.
\end{corollary}

Thus for quasilinear first order systems, one can find $C^2$ conserved quantities in two steps: First, find conservation law multipliers of $\bb F$ using Corollary \ref{EulerCond}. Second, compute the corresponding conserved quantities $\bb \psi$ using \eqref{multCond1}. The use of the Euler operator to find conserved quantities applies to much more general systems, such as normal, nondegenerate PDE systems. For more details, see \cite{blum10Ay, olve86Ay}.

\subsection{Examples of conservation law multipliers}\text{ }\label{exCLMult}
\begin{example}[Hamiltonian system]
\text{ }\\ 
One classical time-independent conserved quantity of an autonomous system is the energy $H(\bb q, \bb p)$ of a Hamiltonian system,
\begin{equation}
\bb F(\bb q, \bb p, \dot{\bb p}, \dot{\bb q}) := 
\begin{pmatrix} \dot{\bb q} \\ \dot{\bb p} \end{pmatrix}-J\begin{pmatrix} \partial_{\bb p} H(\bb q,\bb p) \\ \partial_{\bb q} H(\bb q,\bb p)\end{pmatrix} = \bb 0, \label{ex:HamSys}
\end{equation} where $\bb q\in \mathbb{R}^n$ are the generalized coordinates, $\bb p\in \mathbb{R}^n$ are the generalized momenta and $J$ is the $(2n)\times(2n)$ skew-symmetric matrix
\begin{equation}
J=\begin{pmatrix}0 & I_n \\ -I_n & 0\end{pmatrix}, \text{ with the } n\times n\text{ identity matrix } I_n. \label{skewSymMat}
\end{equation}
Since $H$ is a conserved quantity of \eqref{ex:HamSys}, it follows from \eqref{multCond1} that the $2n\times 1$ matrix 
\begin{equation}
\Lambda(\bb q,\bb p)=\begin{pmatrix} \partial_{\bb q} H(\bb q,\bb p) & \partial_{\bb p} H(\bb q,\bb p) \end{pmatrix} \label{HamMult}
\end{equation} is the multiplier associated with the energy $H$. Indeed, for any $\bb q, \bb p\in C^1(I\rightarrow \mathbb{R}^n)$,
\[
\Lambda(\bb q,\bb p)\bb F(\bb q, \bb p, \dot{\bb p}, \dot{\bb q}) &= \begin{pmatrix} \partial_{\bb q} H(\bb q,\bb p) & \partial_{\bb p} H(\bb q,\bb p) \end{pmatrix}\left(\begin{pmatrix} \dot{\bb q} \\ \dot{\bb p} \end{pmatrix}-J\begin{pmatrix} \partial_{\bb p} H(\bb q,\bb p) \\ \partial_{\bb q} H(\bb q,\bb p)\end{pmatrix}\right) \\
&= \partial_{\bb q} H(\bb q,\bb p)\cdot\dot{\bb q}+\partial_{\bb p} H(\bb q,\bb p)\cdot\dot{\bb p} \\
&= D_t H(\bb q, \bb p).
\] Moreover, \eqref{multCond2} is also satisfied since $\Lambda \bb f = 0 = -\partial_{t} H$, where $\bb f:=J\begin{pmatrix} \partial_{\bb p} H & \partial_{\bb q} H\end{pmatrix}^T$.
\end{example}

\begin{example}[Damped Harmonic Oscillator] \label{dhoSys}
\text{ }\\
The damped harmonic oscillator written as an autonomous system is given by,
\begin{equation}
\bb F(x,y,\dot{x},\dot{y}) := 
\begin{pmatrix} \dot{x} \\ \dot{y} \end{pmatrix}-\begin{pmatrix} y \\ -\frac{1}{m} \left(\gamma y+\kappa x\right)\end{pmatrix} = \bb 0, \label{dhoSysEqn}
\end{equation} where $m$ is the mass of an object attached to a spring with the spring constant $\kappa$ and the damping coefficient $\gamma$. In \cite{wan15a}, the time-dependent conserved quantity $\psi(t,x,y)$,
\begin{equation*}
\psi(t,x,y) := \frac{e^{\frac{\gamma}{m}t}}{2}\left(my^2+\gamma x y + \kappa x^2\right), \label{dhoCQ}
\end{equation*} was found for the damped harmonic oscillator. Here we look for the corresponding multiplier as a first order system \eqref{dhoSysEqn}. By \eqref{multCond1},
\begin{equation*}
\Lambda(t,x,y) = \begin{pmatrix} \frac{\partial \psi}{\partial x}(t,x,y) & \frac{\partial \psi}{\partial y}(t,x,y) \end{pmatrix} = \begin{pmatrix} e^{\frac{\gamma}{m}t} \left(\kappa x+\frac{\gamma}{2}y\right)& e^{\frac{\gamma}{m}t} \left(\frac{\gamma}{2}x+my\right)\end{pmatrix}. \label{dhoMult}
\end{equation*} It follows that for any $x, y\in C^1(I\rightarrow \mathbb{R})$
\[
&\Lambda(t,x,y)\bb F(x,y,\dot{x},\dot{y}) = e^{\frac{\gamma}{m}t}\left(\left(\kappa x+\frac{\gamma}{2}y\right)(\dot{x}-y)+\left(my+\frac{\gamma}{2}x\right)\left(\dot{y}+\frac{\gamma}{m}y+\frac{\kappa}{m}x\right)\right)\\
&\hskip 4mm= e^{\frac{\gamma}{m}t}\left(\frac{\gamma}{2m}\left(m y^2+\gamma x y + \kappa x^2\right)+my\dot{y}+\frac{\gamma}{2}(\dot{x}y+x\dot{y})+\kappa x\dot{x}\right) = D_t \psi(t,x,y).
\]
Moreover, one can verify that \eqref{multCond2} is indeed satisfied, since
\[
\Lambda(t,x,y)\begin{pmatrix} y \\ -\frac{1}{m} \left(\gamma y+\kappa x\right)\end{pmatrix} &= -\frac{e^{\frac{\gamma}{m}t}}{2}\frac{\gamma}{m}\left(m y^2+\gamma xy+\kappa x^2\right) = -\partial_t\psi(t,x,y).
\]
\end{example}

\subsection{Local solvability of $\bb f$}
\text{ }\\
There is another form of condition \eqref{multCond2} which will be useful in application. For this, we need some mild assumptions on the conserved quantity $\bb \psi$. 

\begin{definition}
Let $\bb \psi\in C^1(I\times U\rightarrow \mathbb{R}^m)$ be a conserved quantity of $\bb F$. The components of $\bb \psi$ are linearly independent on $I\times U$ if $\partial_{\bb x} \bb \psi$ has full row rank on $I\times U$.
\end{definition}
Note that since $\bb x$ has at most $n$ components, the Jacobian $\partial_{\bb x} \bb \psi$ can have full row rank only if $m\leq n$. We now derive a theorem on local solvability of components of $\bb f$ using conditions \ref{multCond1} and \ref{multCond2}.

\begin{theorem}[Local solvability of $\bb f$]\label{localSolveF}
Let $n,m \in \mathbb{N}$ with $1\leq m\leq n$ and let $\bb \psi\in C^1(I\times U\rightarrow \mathbb{R}^m)$ be a linearly independent conserved quantity of $\bb F$ on $I\times U$. \\If $m\leq n-1$, then for any $(s,\bb y)\in I\times U$, there exist open balls $B_R(s)\times B_R(\bb y)\subset I\times U$ around $(s,\bb y)$ and a $n\times n$ permutation matrix $P$ such that for $(t,\bb x)\in B_R(s) \times B_R(\bb y)$,
\[
(\Lambda P^T)(t,\bb x) &= \begin{pmatrix}\tilde{\Lambda}(t,\bb x)&\Sigma(t,\bb x)\end{pmatrix}, \\
(P\bb f)(t,\bb x) &= \begin{pmatrix}\tilde{\bb f}(t,\bb x) \\ \bb g(t,\bb x)\end{pmatrix},
\] where $\tilde{\Lambda} \in C(B_R(s) \times B_R(\bb y)\rightarrow M_{m\times m}(\mathbb{R}))$ is invertible, $\Sigma \in C(I\times U\rightarrow$ \\$M_{m\times (n-m)}(\mathbb{R}))$ and 
$\tilde{\bb f} \in C^{p-1}(I\times U\rightarrow\mathbb{R}^m), \bb g \in C^{p-1}(I\times U\rightarrow\mathbb{R}^{n-m})$ satisfying,
\begin{align}
\tilde{\bb f}(t,\bb x) = -\left[\tilde{\Lambda}(t,\bb x)\right]^{-1}\left(\frac{}{}\partial_t\bb \psi(t,\bb x)+\Sigma(t,\bb x) \bb g(t,\bb x)\right). \label{contSolveF}
\end{align} In the case $m=n$, $\bb f$ can be solved globally on $I\times U$ given by,
\begin{equation}
\bb f(t,\bb x) = -\left[\Lambda(t,\bb x)\right]^{-1}\partial_t\bb \psi(t,\bb x). \label{contSolveFFull}
\end{equation}
\end{theorem}
\begin{proof} Let $1\leq m\leq n-1$ and fix any $(s,\bb y)\in I\times U$. Since $\partial_{\bb x} \bb \psi$ has full row rank on $I\times U$ and $\Lambda = \partial_{\bb x} \bb \psi$ by condition \eqref{multCond1}, there must be a $m\times m$ minor $\tilde{\Lambda}$ of $\Lambda$ such that $\det(\tilde{\Lambda}(s,\bb y))\neq 0$. By continuity of $\tilde{\Lambda}$ and the determinant function, there exists open balls $B_R(s) \times B_R(\bb y)\subset I\times U$ around $(s,\bb y)$ so that $\det(\tilde{\Lambda}(t,\bb x))\neq 0$ for all $(t,\bb x)\in B_R(s) \times B_R(\bb y)$. Reorder the columns of $\Lambda$ so that the invertible minor $\tilde{\Lambda}$ is on the first $m$ columns; i.e.~there is a permutation matrix $P$ such that $\Lambda P^T = \begin{pmatrix}\tilde{\Lambda}&\Sigma\end{pmatrix}$ with $\tilde{\Lambda}\in M_{m\times m}(\mathbb{R})$ invertible on $B_R(s) \times B_R(\bb y)$. Thus by condition \eqref{multCond2}, for $(t,\bb x)\in B_R(s) \times B_R(\bb y)$,
\[
&-\partial_t \bb \psi = \Lambda \bb f = \left(\Lambda P^T\right) \left(P\bb f\right) = \begin{pmatrix}\tilde{\Lambda}&\Sigma\end{pmatrix} \begin{pmatrix}\tilde{\bb f} \\ \bb g\end{pmatrix}= \tilde{\Lambda}\tilde{\bb f} + \Sigma \bb g, \text{ with }\tilde{\Lambda} \text{ invertible}.
\] Solving for $\tilde{\bb f}$ by inverting $\tilde{\Lambda}$ shows \eqref{contSolveF}. In the case when $m=n$, $\Lambda=\partial_{\bb x} \bb \psi$ is an invertible square matrix on $I\times U$, which implies \eqref{contSolveFFull}.
\end{proof}

Interestingly, for \emph{non-trivial} first order quasilinear ODEs, \eqref{contSolveFFull} implies there can be at most $n-1$ linearly independent components of the form $\bb \psi(\bb x)$\footnote{This was remarked in \cite{wan16} for autonomous systems with time-independent conserved quantities.}.
\begin{corollary}
Let $\bb \psi\in C^1(U\rightarrow \mathbb{R}^m)$ be a time independent conserved quantity of $\bb F$ with linearly independent components on $U$. Then either $m\leq n-1$, or $\bb f(t,\bb x) = 0$ for all $(t, \bb x)\in I\times U$, in which case $\bb F(t,\bb x, \dot{\bb x}):=\dot{\bb x}(t)=0$ or $\bb x(t)=\bb x_0$.
\end{corollary}
\begin{proof}
If suffices to show the case for $m=n$, as $m\leq n$ by hypothesis of linear independence. Since $\partial_t \bb \psi=\bb 0$, \eqref{contSolveFFull} implies $ \bb f(t,\bb x)=\bb 0$ for all $(t,\bb x)\in I\times U$.
\end{proof}

\section{The multiplier method}
\label{sec:multMethod}
\text{ }\\
Combining the theories developed so far for conservation law multipliers of Section~\ref{sec:multTheory} and the divided difference calculus introduced in Appendix \ref{sec:divDiff}, we now demonstrate a systematic way to construct conservative schemes for first order quasilinear ODEs. 

\subsection{Sufficient conditions for conservative schemes of arbitrary order}
Let $\{t^k\in \mathbb{R}\}_{k\in \mathbb{N}}$ be a discrete set of time steps with $t^k<t^{k+1}$ such that there exists a largest time step size $\tau = \sup_{k\in \mathbb{N}} (t^{k+1}-t^k)<\infty$. We denote discrete approximations of $\bb x(t^k)\in U$ as $\bb x^k \in U$. Specifically, we will be focusing on (nonlinear) multi-step methods.

\begin{definition}
Let $r\in\mathbb{N}$ and $W$ be a finite dimensional normed vector space. $f^\tau$ is called a $r$-step function if $f^\tau\colon I\times U^{r+1}\rightarrow W$, where $U^{r+1}$ is the Cartesian product of $r+1$ copies of $U$. The value of $f^\tau$ at $(t^k, \bb x^{k+1}, \dots, \bb x^{k-r+1})\in I\times U^{r+1}$ is denoted by $f^\tau(t^k,\bb x^{k+1},\dots,\bb x^{k-r+1})\in W$.
\end{definition}

\begin{definition}
Let $p,q \in \mathbb{N}$. A $r$-step function $f^\tau\colon C^{p+q}(I\times U^r \rightarrow W)$ is consistent of order $q$ to a function $f\in C^{p+q}(I\times U\times U^{(1)}\times\cdots\times U^{(p)}\rightarrow W)$ if for any $\bb x\in C^{p+q}(I\times U\rightarrow U)$, there exists a constant $C_f>0$ independent of $\tau$ so that
\[
\norm{f(t^k,\bb x(t^k),\dots, {\bb x}^{(p)}(t^k))-f^\tau(t^k,\bb x(t^{k+1}),\dots, \bb x(t^{k-r+1}))}_W \leq C_{f}\norm{\bb x}_{C^{p+q}(I^k)} \tau^q,
\]where $I^k:=[t^{k-r+1},t^{k+1}]$ and $\displaystyle \norm{\bb x}_{C^{r}(I^k)} := \max_{0\leq i\leq r} \norm{\bb x^{(i)}}_{L^\infty(I^k)}$. If so, we simply write $f^\tau = f + \mathcal{O}(\tau^q)$.
\end{definition}

In the following part, $W$ is either $\mathbb{R}^m$ with the usual Euclidean norm or $M_{m\times n}(\mathbb{R})$ with the operator norm. Before stating the main theorem for constructing conservative schemes, we need a few more definitions.

\begin{definition}
Let $\bb F^\tau$ be a consistent $r$-step function to $\bb F$ and $\bb \psi^\tau$ be a consistent $(r-1)$-step function to $\bb \psi$. Denote $\bb \psi^\tau_k := \bb \psi^\tau(t^{k}, \bb x^{k},\dots, \bb x^{k-r+1})$. We say the $r$-step method $\bb F^\tau$ is conservative in $\bb \psi^\tau$ if $\bb \psi^\tau_{k+1}=\bb \psi^\tau_k$, whenever $\bb x^{k+1}$ satisfies $\bb F^\tau(t^k, \bb x^{k+1},\dots, \bb x^{k-r+1})=\bb 0$.
\end{definition}

\begin{definition}
Let $D^\tau_t \bb \psi$ be a consistent $r$-step function to $D_t \bb \psi$ and $\bb \psi^\tau$ be a consistent $(r-1)$-step function to $\bb \psi$. We say that $D^\tau_t \bb \psi$ is constant-compatible with $\bb \psi^\tau$ if $D^\tau_t \bb \psi(t^k, \bb x^{k+1},\dots, \bb x^{k-r+1})=\bb 0$ implies $\bb \psi^\tau_{k+1}=\bb \psi^\tau_k$,
\end{definition}

In other words, constant-compatibility means the discrete total derivative $D^\tau_t \bb \psi$ preserves the vanishing derivative of constant functions. We now prove a key theorem for constructing conservative schemes of arbitrary orders, which generalizes the case of autonomous systems with time-independent conserved quantities shown in \cite{wan16}.

\begin{theorem}[Conservative discretizations for quasilinear first order ODEs]\text{ }\\ \label{mainTheorem}Suppose $\bb f^\tau, D^\tau_t \bb x, D^\tau_t \bb \psi, \partial^\tau_t \bb \psi, \Lambda^\tau$ are $r$-step functions and consistent of order $q$ respectively to $\bb f, \dot{\bb x}, D_t\bb \psi, \partial_t \bb \psi, \Lambda$, where $\Lambda$ is a conservation law multiplier of $\bb F$ associated with the conserved quantity $\bb \psi$. Assume $D^\tau_t \bb \psi$ is constant-compatible with a $(r-1)$-step function $\bb \psi^\tau$. Also assume $\bb f^\tau, D^\tau_t \bb x, D^\tau_t \bb \psi, \partial^\tau_t \bb \psi, \Lambda^\tau$ satisfy
\begin{subequations}
\begin{align}
\Lambda^\tau D^\tau_t \bb x &= D^\tau_t \bb \psi - \partial^\tau_t \bb \psi, \label{discCond1}\\
\Lambda^\tau \bb f^\tau &= -\partial^\tau_t \bb \psi. \label{discCond2}
\end{align}
\end{subequations}
Then the $r$-step method defined by
\begin{align} \label{discFormula}
\bb F^\tau(t^k, \bb x^{k+1},\dots,\bb x^{k-r+1}):= &D^\tau_t \bb x(t^k, \bb x^{k+1},\dots,\bb x^{k-r+1})\\
&\hskip 4mm - \bb f^\tau(t^k, \bb x^{k+1},\dots,\bb x^{k-r+1}) = \bb 0, \nonumber
\end{align} is consistent of at least $q$-th order to $\bb F$ and $\bb F^\tau$ is conservative in $\bb \psi^\tau$. Moreover, for any $\bb x\in C^q(I\rightarrow\mathbb{R}^n)$,
\begin{subequations}
\begin{align}
\Lambda^\tau D^\tau_t \bb x -D^\tau_t \bb \psi +\partial^\tau_t \bb \psi &= \mathcal{O}(\tau^q), \label{consCond1} \\
\Lambda^\tau \bb f^\tau+\partial^\tau_t \bb \psi &= \mathcal{O}(\tau^q). \label{consCond2}
\end{align}
\end{subequations}
\end{theorem}
\begin{proof} It is clear from the triangle inequality that if $D^\tau_t \bb x= \dot{\bb x}+\mathcal O(\tau^q)$ and $\bb f^\tau= \bb f+\mathcal O(\tau^q)$, then $\bb F^\tau= \bb F+\mathcal O(\tau^q)$. Let $\bb x^{k+1}$ be a solution to $\bb F^\tau=0$. Then by \eqref{discCond1} and \eqref{discCond2},
\[
0 = \Lambda^\tau(D^\tau_t \bb \psi - \bb f^\tau) = \left(D^\tau_t \bb \psi - \partial^\tau_t \bb \psi\right) +\partial^\tau_t \bb \psi = D^\tau_t \bb \psi.
\] Since $D^\tau_t \bb \psi$ is constant-compatible with $\bb \psi^\tau$, this implies $\bb \psi^\tau_{k+1}=\bb \psi^\tau_k$. In other words, $\bb F^\tau$ is conservative in $\bb \psi^\tau$. 
To show \eqref{consCond1}, for any $\bb x\in C^1(I\rightarrow\mathbb{R}^n)$ note that $\Lambda \dot{\bb x} = \partial_{\bb x} \bb \psi \cdot\dot{\bb x}$ by \eqref{multCond1} and $\partial_{\bb x} \bb \psi \cdot \dot{\bb x} = D_t \bb \psi -\partial_t\bb \psi$ by the chain rule. Thus,
\[
&\norm{\Lambda^\tau D^\tau_t \bb x-D^\tau_t \bb \psi +\partial^\tau_t \bb \psi} = \norm{\left(\Lambda^\tau D^\tau_t \bb x-\Lambda \dot{\bb x}\right)+\left(\partial_{\bb x} \bb \psi \cdot\dot{\bb x}-D^\tau_t \bb \psi +\partial^\tau_t \bb \psi\right)}\\
&\hskip 2mm\leq \norm{\left(\Lambda^\tau-\Lambda\right) D_t^\tau \bb x + \Lambda(D_t^\tau \bb x-\dot{\bb x})}+ \norm{\left(D_t \bb \psi-D_t^\tau \bb \psi\right) + \left(\partial^\tau_t \bb \psi-\partial_t \bb \psi\right)}\\
&\hskip 2mm\leq \norm{\Lambda^\tau-\Lambda}(\norm{\dot{\bb x}}+\norm{D_t^\tau \bb x-\dot{\bb x}})+\norm{\Lambda}\norm{D_t^\tau \bb x-\dot{\bb x}}+ \norm{D_t \bb \psi-D_t^\tau \bb \psi}\\
&\hskip 6mm+\norm{\partial^\tau_t \bb \psi-\partial_t \bb \psi} = \mathcal{O}(\tau^q),
\] where the last step follows from the order $q$ consistency of $\Lambda^\tau, D_t^\tau \bb x, \partial_t^\tau \bb \psi, D_t^\tau \bb \psi$ and that $\norm{\dot{\bb x}(t)}, \norm{\Lambda(t, \bb x(t))}$ are bounded uniformly on $t\in I^k$ by continuity of $\bb x$ and $\Lambda$. A similar estimate can be carried out to show \eqref{consCond2}.
\end{proof}
\begin{remark} Although we have only shown consistency of at least order $q$, it is possible for \eqref{discFormula} to be of higher order than $q$, as Example \ref{sec:RBRotDisc} in Section \ref{sec:CM_examples} illustrates.
\end{remark}

\begin{remark}The average vector field method was introduced in \cite{QuiMcL08} as an energy-preserving discretization for Hamiltonian systems. This method can be viewed as a special case of the conditions \eqref{discCond1} and \eqref{discCond2} applied to Hamiltonian systems. In particular, since $\psi(t,\bb x):=H(\bb x)$ is time independent, the average vector field method is equivalent to the following choices of $D^\tau_t \bb x, D^\tau_t H, \partial^\tau_t H, \bb f^\tau, \Lambda^\tau$:
\[
D^\tau_t \bb x &:= \frac{\bb x^{k+1}-\bb x^k}{\tau}, & H^\tau(\bb x^k) &:= H(\bb x^k)\\
D^\tau_t H &:= \frac{H(\bb x^{k+1})-H(\bb x^k)}{\tau} & \bb f^\tau &:= J\partial_{\bb x}^\tau H\\
\partial^\tau_t H &:= 0 & \Lambda^\tau &:= \partial_{\bb x}^\tau H
\] where $J$ is the $(2n)\times(2n)$ skew-symmetric matrix from \eqref{skewSymMat} and 
\[
\partial_{\bb x}^\tau H(\bb x^{k+1}, \bb x^k):= \int_0^1 \partial_{\bb x}H\left(s(\bb x^{k+1}-\bb x^{k})+\bb x^{k}\right)ds,
\] which is consistent to $\Lambda=\partial_{\bb x} H$ from \eqref{HamMult}. Thus, \eqref{discCond1} and \eqref{discCond2} are satisfied since
\[
\Lambda^\tau D^\tau_t \bb x &=  \frac{1}{\tau}\int_0^1 \partial_{\bb x}H\left(s(\bb x^{k+1}-\bb x^{k})+\bb x^{k}\right)\cdot(\bb x^{k+1}-\bb x^k)ds \\
&= \frac{1}{\tau}\int_0^1 \frac{d}{ds} H\left(s(\bb x^{k+1}-\bb x^{k})+\bb x^{k}\right)ds\\
&= \frac{H(\bb x^{k+1})-H(\bb x^k)}{\tau} = D^\tau_t H = D^\tau_t H - \partial^\tau_t H, \\
\Lambda^\tau \bb f^\tau &= \partial_{\bb x}^\tau H \cdot J\partial_{\bb x}^\tau H = 0 = -\partial^\tau_t H.
\]
In Section \ref{firstOrdCM}, conservative discretizations for quasilinear first order systems are derived using divided differences. In particular, in contrast to the average vector field method, the multiplier method does not require computation of integrals and can be directly applied to non-Hamiltonian systems.
\end{remark}

\subsection{Local solvability of $\bb f^\tau$}

Note that conditions \eqref{discCond1} and \eqref{discCond2} are discrete analogues of conditions \eqref{multCond1} and \eqref{multCond2}. Moreover, given a consistent discrete multiplier $\Lambda^\tau$ to $\Lambda$, we show that condition \eqref{discCond2} can be satisfied locally in $I\times U$ for sufficiently small $\tau$ using the local invertibility of discrete multipliers of Lemma \ref{localInverse} presented in the Appendix \ref{sec:localInverse}. 

\begin{theorem}[Local solvability of $\bb f^\tau$]
\label{discLocalSolve}
Let $(s,\bb y)\in I\times U$ and $\Sigma, \tilde{\bb f}, \bb g, P$ and an invertible matrix $\tilde{\Lambda}$ be as given by Theorem \ref{localSolveF}. Suppose $\Lambda^\tau, \partial_t^\tau \bb \psi, \bb g^\tau$ are consistent of order $q$ to $\Lambda, \partial_t \bb \psi, \bb g$ and $\{\tilde{\Lambda}^\tau\}_{0<\tau<\tau_0}$ is equicontinuous on $I\times U^{r+1}$. \\ If $1\leq m\leq n-1$, define $\tilde{\Lambda}^\tau \in M_{m\times m}(\mathbb{R})$ and $\Sigma^\tau \in M_{(n-m)\times m}(\mathbb{R})$ by,
\begin{align}
\Lambda^\tau P^T := \begin{pmatrix} \tilde{\Lambda}^\tau & \Sigma^\tau \end{pmatrix}.
\end{align} Then for sufficiently small $\tau$ and some $r>0$, $\tilde{\Lambda}^\tau$ is invertible on $B_r(s)\times B_r(\bb y)\times \cdots \times B_r(\bb y)\subset I\times U^{r+1}$ and
\begin{align}
\bb f^\tau := P^T\begin{pmatrix} -\left[\tilde{\Lambda}^\tau\right]^{-1}\left(\partial^\tau_t \bb \psi + \Sigma^\tau \bb g^\tau \right) \\ \bb g^\tau \end{pmatrix}, \label{discFEqn}
\end{align} satisfies condition \eqref{discCond2} and is consistent of order $q$ to $\bb f$.
In the case $m=n$, for sufficiently small $\tau$ and some $r>0$, $\Lambda^\tau$ is invertible on $B_r(s)\times B_r(\bb y)\times \cdots \times B_r(\bb y)\subset I\times U^{r+1}$ and
\begin{equation}
\bb f^\tau= -\left[\Lambda^\tau\right]^{-1}\partial_t^\tau\bb \psi, \label{discFEqnFull}
\end{equation} is consistent of order $q$ to $\bb f$.
\end{theorem}
\begin{proof} First, consider the case when $m=n$. By Theorem \ref{localSolveF}, $\Lambda{(t,\bb x)}$ is invertible on some open balls $B_R(s)\times B_R(\bb y)\subset I\times U$. Without loss of generality, we can choose $R$ so that $B_R(s)\times B_R(\bb y)$ is closed. Thus, by local invertibility of discrete multiplier from Lemma \ref{localInverse}, for sufficiently small $\tau$ and $r\leq R$, $\Lambda^\tau$ is invertible on $B_r(s)\times B_r(\bb y)\times \cdots \times B_r(\bb y)\subset I\times U^{r+1}$ with the uniform bound $\norm{[\tilde{\Lambda}^\tau]^{-1}}\leq C$ for some constant $C>0$ independent of $\tau$. Thus, \eqref{discFEqnFull} is well-defined. To show $\bb f^\tau=\bb f +\mathcal{O}(\tau^q)$, note from \eqref{contSolveFFull} and for any $\bb x\in C(B_r(s)\rightarrow B_r(\bb y))$,
\[
&\norm{-\left[\Lambda^\tau\right]^{-1}\partial_t^\tau\bb \psi-\bb f} = \norm{(\Lambda^{-1}-\left[\Lambda^\tau\right]^{-1})\partial_t^\tau\bb \psi-\Lambda^{-1}(\partial_t^\tau\bb \psi-\partial_t\bb \psi)}\\
&\hskip 8mm\leq \norm{\left[\Lambda^\tau\right]^{-1}(\Lambda^\tau-\Lambda)\Lambda^{-1}\partial_t^\tau\bb \psi}+\norm{\Lambda^{-1}(\partial_t^\tau\bb \psi-\partial_t\bb \psi)}\\
&\hskip 8mm\leq \norm{\left[\Lambda^\tau\right]^{-1}}\norm{\Lambda^\tau-\Lambda}\norm{\Lambda^{-1}}(\norm{\partial_t\bb \psi}+\norm{\partial_t^\tau\bb \psi-\partial_t\bb \psi})+\norm{\Lambda^{-1}}\norm{(\partial_t^\tau\bb \psi-\partial_t\bb \psi)}\\
&\hskip 8mm = \mathcal{O}(\tau^q).
\] The last step follows from the uniform bound of $\norm{\left[\Lambda^\tau\right]^{-1}}$ on $B_r(s)\times B_r(\bb y)\times \cdots \times B_r(\bb y)$, the order $q$ consistency of $\Lambda^\tau, \partial_t^\tau$ and that $\norm{\partial_t \bb \psi(t,{\bb x}(t))}, \norm{\Lambda^{-1}(t, \bb x(t))}$ are bounded uniformly on $t\in B_r(s)$ by continuity of $\bb x$, $\partial_t \bb \psi$ and $\Lambda^{-1}$. \\
The case when $1\leq m\leq n-1$ follows similarly. We only highlight the main steps. Since $\tilde{\Lambda}$ is invertible on some closed balls $B_R(s)\times B_R(\bb y)$ by \eqref{localSolveF}, Lemma \ref{localInverse} implies $\tilde{\Lambda}^\tau$ is invertible on $B_r(s)\times B_r(\bb y)\times \cdots \times B_r(\bb y)$ with the uniform bound $\norm{[\tilde{\Lambda}^\tau]^{-1}}\leq C$ for some constant $C>0$ independent of $\tau$.
Then, $\bb f^\tau$ from \eqref{discFEqn} is well-defined and \eqref{discCond2} is satisfied since,
\[
\Lambda^\tau \bb f^\tau = (\Lambda^\tau P^T) (P\bb f^\tau) = \begin{pmatrix} \tilde{\Lambda}^\tau & \Sigma^\tau \end{pmatrix}\begin{pmatrix} -\left[\tilde{\Lambda}^\tau\right]^{-1}\left(\partial^\tau_t \bb \psi + \Sigma^\tau \bb g^\tau \right)\\ \bb g^\tau \end{pmatrix} = -\partial_t^\tau \bb \psi.
\] 
To show $\bb f^\tau=\bb f +\mathcal{O}(\tau^q)$, note from \eqref{contSolveF} and for any $\bb x\in C(B_r(s)\rightarrow B_r(\bb y))$,
\[
&\norm{-\left[\tilde{\Lambda}^\tau\right]^{-1}\left(\partial^\tau_t \bb \psi + \Sigma^\tau \bb g^\tau\right)-\tilde{\bb f}} = \norm{\tilde{\Lambda}^{-1}\left(\partial_t \bb \psi + \Sigma \bb g\right)-\left[\tilde{\Lambda}^\tau\right]^{-1}\left(\partial^\tau_t \bb \psi + \Sigma^\tau \bb g^\tau\right)} \\
&\hskip 4mm= \norm{\tilde{\Lambda}^{-1}\left(\frac{}{}\left(\partial_t\bb \psi-\partial_t^\tau \bb \psi\right)+\left(\Sigma \bb g-\Sigma^\tau\bb g^\tau\right)\right)+\tilde{\Lambda}^{-1}(\tilde{\Lambda}^\tau-\tilde{\Lambda})[\tilde{\Lambda}^\tau]^{-1}\left(\partial_t^\tau\bb \psi+\Sigma^\tau \bb g^\tau\right)}\\
&\hskip 4mm\leq \norm{\tilde{\Lambda}^{-1}}\left(\norm{\partial_t \bb \psi-\partial^\tau_t \bb \psi}+\norm{\Sigma \bb g-\Sigma^\tau\bb g^\tau}\right)\\
&\hskip 8mm+\norm{\tilde{\Lambda}^{-1}}\norm{\tilde{\Lambda}^\tau-\tilde{\Lambda}}\norm{[\tilde{\Lambda}^\tau]^{-1}}\norm{\partial_t^\tau\bb \psi+\Sigma^\tau \bb g^\tau}=\mathcal{O}(\tau^q),
\]
where the last step follows from the uniform bound of $\norm{[\tilde{\Lambda}^\tau]}^{-1}$ on $B_r(s)\times B_r(\bb y)\times \cdots \times B_r(\bb y)$, from the order $q$ consistency of $\Sigma^\tau, \tilde{\Lambda}^\tau, \partial_t^\tau\bb \psi, \bb g^\tau$ and from the continuity of $\Sigma, \tilde{\Lambda}, \partial_t\bb \psi, \bb g, \tilde{\Lambda}^{-1}, 
\bb x$ so that
\[
\norm{\Sigma \bb g-\Sigma^\tau\bb g^\tau}&\leq \norm{\Sigma}\norm{\bb g-\bb g^\tau}+\norm{\Sigma-\Sigma^\tau}\left(\norm{\bb g}+\norm{\bb g^\tau-\bb g}\right) = \mathcal{O}(\tau^q) \\
\norm{\partial_t^\tau\bb \psi+\Sigma^\tau \bb g^\tau} &\leq \norm{\partial_t\bb \psi+\Sigma \bb g}+\norm{\partial_t \bb \psi-\partial^\tau_t \bb \psi}+\norm{\Sigma}\norm{\bb g-\bb g^\tau}\\
&\hskip 4mm+\norm{\Sigma-\Sigma^\tau}\left(\norm{\bb g}+\norm{\bb g^\tau-\bb g}\right) \leq C_1, \\
\] Combining with the hypothesis that $\bb g^\tau = \bb g+\mathcal{O}(\tau^q)$, it follows that $\bb f^\tau=\bb f+\mathcal{O}(\tau^q)$.
\end{proof}
\begin{remark} Although the local solvability of $\bb f^\tau$ generally results in expressions defined on a smaller domain $B_r(s)\times B_r(\bb y)\times \cdots \times B_r(\bb y)\subset I\times U^{r+1}$, in practice, due to cancellations that can occur with $\partial^\tau_t \bb \psi + \Sigma^\tau \bb g^\tau$ and $[\tilde{\Lambda}^\tau]^{-1}$, it is possible for the final form of the discretization $\bb f^\tau$ to be defined in the original domain; see examples in Section \ref{sec:CM_examples}.
\end{remark}

\subsection{Construction of first order conservative schemes}
\label{firstOrdCM}\text{ }\\
The result of Theorem \ref{mainTheorem} shows that conservative schemes of arbitrary order can be constructed provided the discrete versions of $\Lambda, \bb f, \dot{\bb x}, D_t\bb \psi, \partial_t \bb \psi$ satisfy \eqref{discCond1} and \eqref{discCond2} and $D_t\bb \psi$ is constant-compatible to $\bb \psi^\tau$. In this section, we derive first order conservative schemes using conditions \eqref{discCond1} and \eqref{discCond2} and the divided difference calculus developed in Appendix \ref{sec:divDiff}.

Using the divided differences defined in Appendix \ref{sec:divDiff}, we define the following discrete quantities:
\begin{align}
\bb \psi^\tau(t^k,\bb x^k) &:=  \bb \psi(t^k, \bb x^k), \label{1stOrdDisc1} \\
D_t^\tau \bb x(t^k,\bb x^{k+1}, \bb x^k)  &:= \frac{\Delta \bb x}{\Delta t}(\bb X^{\bb k}) =  \frac{\bb x^{k+1}-\bb x^k}{t^{k+1}-t^k}  \label{1stOrdDisc2}\\
D_t^\tau \bb \psi(t^k,\bb x^{k+1}, \bb x^k) &:= \frac{\Delta \bb \psi}{\Delta t}(\bb X^{\bb k}) = \frac{\bb \psi(t^{k+1}, \bb x^{k+1})-\bb \psi(t^k, \bb x^k)}{t^{k+1}-t^k} \label{1stOrdDisc3}
\end{align}
Thus, it immediately follows that $D_t^\tau \bb \psi=\bb 0$ implies $\bb \psi^\tau_{k+1}=\bb \psi^\tau_k$, i.e. $D_t^\tau \bb \psi$ is constant-compatible with $\bb \psi^\tau$. It remains to define $\Lambda^\tau$ and $\bb f^\tau$ such that conditions \eqref{discCond1} and \eqref{discCond2} are satisfied.

For any $\sigma\in S_{n+1}$ permutation of $\{0,\dots, n\}$, with the sequence of multi-indices $\bb{v}_{i+1}=\bb{v}_{i}+\bb e_{\sigma(i)}\in \mathbb{N}^{n+1}$ and $\bb v_0=\bb 0$, define  
\begin{equation}
\partial_t^\tau \bb \psi(t^k,\bb x^{k+1}, \bb x^k) :=  \frac{\Delta}{\Delta t}\bb \psi(\bb X^{\bb k+\bb v_{\sigma^{-1}(0)}}) = \frac{\Delta_0\bb \psi}{\Delta t}(\bb X^{\bb k+\bb v_{\sigma^{-1}(0)}}) \label{1stOrdDisc4}
\end{equation}
and the discrete multiplier to be
\begin{align}
\Lambda^\tau(t^k,\bb x^{k+1}, \bb x^k) &:= \begin{pmatrix}\frac{\Delta}{\Delta x_1}\bb \psi(\bb X^{\bb k+\bb{v}_{\sigma^{-1}(1)}}) &\cdots &\frac{\Delta}{\Delta x_n}\bb \psi(\bb X^{\bb k+\bb{v}_{\sigma^{-1}(n)}})\end{pmatrix}, \label{discMultPerm}
\nonumber
\end{align}
Then recalling from Theorem \ref{discLocalSolve}, we can define $\bb f^\tau$ locally by \eqref{discFEqn}. Thus, we have the following (at least) first order conservative method for $\bb F$.

\begin{theorem} \label{firstOrdDisc}
Let $1\leq m\leq n-1$. The discrete quantities $\bb f^\tau, D^\tau_t \bb x, D^\tau_t \bb \psi, \partial^\tau_t \bb \psi, \Lambda^\tau$ defined by \eqref{1stOrdDisc2}--\eqref{1stOrdDisc4}, \eqref{discFEqn} and \eqref{discMultPerm} are consistent to first order to $\bb f, \dot{\bb x}, D_t\bb \psi,$ $\partial_t \bb \psi,\Lambda$ respectively and they satisfy the conditions \eqref{discCond1} and \eqref{discCond2}. In other words, the corresponding $1$-step method $\bb F^\tau$ given by \eqref{discFormula} is consistent of at least first order to $\bb F$ and is conservative in $\bb \psi^\tau=\bb \psi$.
\end{theorem}
\begin{proof}
Noting that condition \eqref{discCond2} of Theorem \ref{mainTheorem} can always be (locally) satisfied by $\bb f^\tau$ defined by \eqref{discFEqn} from Theorem \ref{discLocalSolve}, it remains to show the above choices of discretizations satisfy condition \eqref{discCond1}.

Let $(t^k,\bb x^k), (t^{k+1},\bb x^{k+1}) \in I\times U$. In the case when the discrete multiplier is given by \eqref{discMultPerm}, we have by \eqref{1stOrdDisc2}, \eqref{1stOrdDisc4} and the relation \eqref{discDiffPerm} on divided differences,
\[
\Lambda^\tau D_t^\tau\bb x = \sum_{i=1}^n \frac{\Delta}{\Delta x_i} \bb \psi(\bb X^{\bb k+\bb{v}_{\sigma^{-1}(i)}}) \frac{\Delta x_i}{\Delta t} &= \frac{\Delta \bb \psi}{\Delta t}(\bb X^{\bb k}) - \frac{\Delta}{\Delta t} \bb \psi(\bb X^{\bb k+\bb v_{\sigma^{-1}(0)}}) \\
&= D_t^\tau \bb \psi - \partial_t^\tau \bb \psi,
\]
which implies condition \eqref{discCond1}. 
\end{proof}
\begin{remark}
One can also devise a symmetrized version of Theorem \ref{firstOrdDisc} using symmetrized discrete multiplier with symmetrized divided differences of \eqref{discDiffSym}.
\end{remark}

In addition to using \eqref{discFEqn}, there is another general approach to satisfy \eqref{discCond2}. The main idea is to look for ``undetermined consistent terms" of $\bb f^\tau$ to $\bb f$ and use \eqref{discCond2} as a constraint to compute these. This idea is illustrated in the Examples \ref{sec:2LV}, \ref{sec:3bodyDisc} and \ref{sec:dhoDisc} of Section \ref{sec:CM_examples}.

\section{Examples of conservative schemes for dynamical systems}\text{ }\\\label{sec:CM_examples}In this section, we give examples of conservative schemes for dynamical systems using the multiplier method proposed in Section \ref{sec:multMethod}.
Since the multiplier method can be applied directly to quasilinear first order systems, we will illustrate that conservative schemes can be derived in a straightforward manner, regardless of whether the system has a Hamiltonian or, more generally, a Poisson structure. In particular, we will use Theorem \ref{firstOrdDisc} to construct first order conservative schemes for Euler's equations for rigid body rotation, Lotka--Volterra systems, the planar restricted 3-body problem, and the damped harmonic oscillator. For simplicity, we only consider an uniform time step  size $\tau\in\mathbb{R}$, so that $t^{k+1}= t^k+\tau$ for all $k\in \mathbb{N}$. Indeed, Theorem \ref{firstOrdDisc} could also be applied with variable time step sizes, such as for adaptive time-stepping. Often, it will be convenient to denote a specific time average of $f$ as $\bar{f}:=\frac{1}{2}\left(f^{k+1}+f^k\right)$.

\subsection{Rigid body rotation in 3D}\label{sec:RBRotDisc}\text{ }\\
The dynamics of a rigid body's rotation in 3D is governed by Euler's equations
\begin{equation}
\boldsymbol F(\bb \omega,\dot{\bb \omega}):=\begin{pmatrix}
    \dot{\omega}_1 - \dfrac{I_2-I_3}{I_2I_3}\omega_2\omega_3\\
    \dot{\omega}_2 - \dfrac{I_3-I_1}{I_1I_3}\omega_1\omega_3\\
    \dot{\omega}_3 - \dfrac{I_1-I_2}{I_1I_2}\omega_1\omega_2
  \end{pmatrix}=\boldsymbol 0, \label{eulerSys}
\end{equation}
where $I_i$ is the principal moment of inertia and $\omega_i$ is the angular velocity along the $i$-th principal axis of the rigid body. In the absence of external torque, it is well-known that Euler's equations admit two conserved quantities, namely the energy $E(\bb \omega):=\dfrac{\omega_1^2}{I_1}+\dfrac{\omega_2^2}{I_2}+\dfrac{\omega_3^2}{I_3}$ and angular momentum $L(\bb \omega):=\omega_1^2+\omega_2^2+\omega_3^2$. Using \eqref{multCond1} from the correspondence theorem, the conserved quantity $\bb \psi$\footnote{Here, we exclude the degenerate case when $I_1=I_2=I_3$ so that $\bb \psi$ is linearly independent.} and the corresponding $2\times 3$ multiplier matrix $\Lambda(\bb \omega)$ are given by,
\begin{align*}
\bb \psi(\bb \omega) := \begin{pmatrix}
    E(\bb \omega)\\
    L(\bb \omega)\\
  \end{pmatrix}, \hskip 2mm  \Lambda(\bb \omega):=\begin{pmatrix}
\dfrac{\omega_1}{I_1} & \dfrac{\omega_2}{I_2} & \dfrac{\omega_3}{I_3}\\
\omega_1 & \omega_2 & \omega_3\\
  \end{pmatrix}.
\end{align*} We now follow the systematic procedure to construct conservative discretizations for \eqref{eulerSys} outlined in Section \ref{firstOrdCM}. 
Note that $\bb \psi$ is time-independent and so $\partial_t^\tau \bb \psi=\bb 0$. Moreover since its components consist of linear combinations of single variable functions of the form $\omega_i^2$, $\Lambda^\tau$ will be independent of permutations. In particular, for any permutation $\sigma\in S_3$, it follows from the linearity rule of \ref{rule:linearity}, constant rule of \ref{rule:constant} and polynomial rule of \ref{rule:polynomial},
\[
\frac{\Delta}{\Delta \omega_i}\bb \psi(\bb \omega^{k+v_{\sigma^{-1}(i)}}) = \frac{\Delta}{\Delta \omega_i}\begin{pmatrix} 
E(\bb \omega^{k+v_{\sigma^{-1}(i)}}) \\
L(\bb \omega^{k+v_{\sigma^{-1}(i)}})
\end{pmatrix} =
\frac{1}{2}\begin{pmatrix} 
\frac{1}{I_i}\frac{\Delta }{\Delta \omega_i} \omega_i^2\\
\frac{\Delta}{\Delta \omega_i} \omega_i^2
\end{pmatrix} =
\begin{pmatrix} 
\dfrac{\overline{\omega}_i}{I_i}\\
\overline{\omega}_i
\end{pmatrix}.
\]
Thus, the discrete multiplier given by \eqref{discMultPerm} is
\[
\Lambda^\tau(\bb \omega^{k+1},\bb \omega^k) = 
\begin{pmatrix}
\dfrac{\overline{\omega}_1}{I_1} & \dfrac{\overline{\omega}_2}{I_2} & \dfrac{\overline{\omega}_3}{I_3}\\
\overline{\omega}_1 & \overline{\omega}_2 & \overline{\omega}_3\\
\end{pmatrix}.
\]
Since $\bb \psi$ is linearly independent on $U=\{\bb \omega\in\mathbb{R}^3 : \omega_i \neq 0 \text{ for }i=1,2,3\}$, by Theorem \ref{discLocalSolve} for sufficiently small $\tau$, the leftmost $2\times 2$ minor of $\Lambda^\tau$ is invertible on $U$ and so
\[
[\tilde{\Lambda}^\tau(\bb \omega^{k+1},\bb \omega^k)]^{-1} &= \frac{I_1 I_2}{\overline{\omega}_1\overline{\omega}_2(I_2-I_1)}
\begin{pmatrix}
\overline{\omega}_2 & -\dfrac{\overline{\omega}_2}{I_2}\\
-\overline{\omega}_1 & \dfrac{\overline{\omega}_1}{I_1}
\end{pmatrix}, \hskip 2mm \Sigma^\tau(\bb \omega^{k+1},\bb \omega^k) = \begin{pmatrix}
\dfrac{\overline{\omega}_3}{I_3} \\
\overline{\omega}_3
 \end{pmatrix}.
\]
Thus, \eqref{discFEqn} implies for any $g^\tau(\bb \omega^{k+1},\bb \omega^k)$ consistent of first order to $g(\bb \omega) = \frac{I_1-I_2}{I_1I_2}\omega_1\omega_2$,
\[
\tilde{\bb f}^\tau(\omega^{k+1},\omega^k) &= -[\tilde{\Lambda}^\tau]^{-1}\Sigma^\tau g^\tau = \begin{pmatrix}
\frac{(I_2-I_3)I_1}{(I_1-I_2)I_3}\dfrac{\overline{\omega}_3}{\overline{\omega}_1} \\
\frac{(I_3-I_1)I_2}{(I_1-I_2)I_3}\dfrac{\overline{\omega}_3}{\overline{\omega}_2}
 \end{pmatrix} g^\tau(\bb \omega^{k+1},\bb \omega^k).
\]
To simplify $\tilde{\bb f}^\tau$, we make a specific choice for $g^\tau$. Since $g^\tau$ is consistent to $g$, the time averages of $\omega_i$ appearing in components of $\tilde{\bb f}^\tau$ suggest the form,
\[
\bb g^\tau(\bb \omega^{k+1},\bb \omega^k) = \dfrac{I_1-I_2}{I_1I_2}\overline{\omega}_1\overline{\omega}_2, \hskip 2mm \text{ leading to } \hskip 2mm 
\tilde{\bb f}^\tau(\omega^{k+1},\omega^k) = \begin{pmatrix}
\frac{(I_2-I_3)}{I_2I_3}\overline{\omega}_2\overline{\omega}_3 \\
\frac{(I_3-I_1)}{I_1I_3}\overline{\omega}_1\overline{\omega}_3
 \end{pmatrix}.
\]
Thus by Theorem \ref{firstOrdDisc}, equation \eqref{discFormula} gives rise to the conservative discretization,
\[
\boldsymbol F^\tau(\bb \omega^{k+1},\bb \omega^k):=\begin{pmatrix}
\dfrac{\omega_1^{k+1}-\omega_1^k}{\tau}-\dfrac{I_2-I_3}{I_2I_3}\overline{\omega}_2 \overline{\omega}_3\\
\dfrac{\omega_2^{k+1}-\omega_2^k}{\tau}-\dfrac{I_3-I_1}{I_1I_3}\overline{\omega}_1 \overline{\omega}_3\\
\dfrac{\omega_3^{k+1}-\omega_3^k}{\tau}-\dfrac{I_1-I_2}{I_1I_2}\overline{\omega}_1 \overline{\omega}_2\end{pmatrix}=\boldsymbol 0,
\] which conserves $\bb \psi$. Indeed, this is just the midpoint method applied to Euler's equations, which is in fact consistent to second order due to its symmetry\footnote{See Theorem II.3.2 of \cite{hair06Ay}}
. Also, this is not surprising since it is well-known that the midpoint method preserves quadratic invariants \cite{coo87}. Moreover, the final form of the discretizations holds globally for all $\bb \omega\in \mathbb{R}^3$. 

\subsection{Lotka--Volterra systems}\label{sec:LVSysDisc}\text{ }\\
The Lotka--Volterra system is used to model population dynamics of different species \cite{schi03a}. 
We consider the classical 2-species and a degenerate 3-species system.

\subsubsection{Classical 2-species system}\label{sec:2LV}\text{ }\\
The classical 2-species Lotka--Volterra system is given by
\begin{equation}
\boldsymbol F(\bb x):= \begin{pmatrix}
\dot{x} - x(\alpha - \beta y) \\
\dot{y} - y(\delta x - \gamma)\end{pmatrix}= {\bb 0}, \label{LVsys2D}
\end{equation} for some positive constants $\alpha, \beta, \gamma, \delta$ and positive population $x,y$ of two species. It is well-known that \eqref{LVsys2D} has the conserved quantity
\[
V(x,y) := \gamma \log x-\delta x+\alpha \log y- \beta y.
\] 
Thus, by \eqref{multCond1}, the corresponding $1\times 2$ multiplier matrix is
\[
\Lambda(x,y) := \begin{pmatrix}
\dfrac{\gamma}{x}-\delta & \dfrac{\alpha}{y}-\beta
\end{pmatrix}.
\]
Similar to the rigid body example, $V$ is a linear combination of single variable functions and so $\Lambda^\tau$ will be independent of permutations. Indeed, for any permutation $\sigma\in S_2$ of \eqref{discMultPerm}, it follows from the linearity rule of \ref{rule:linearity}, constant rule of \ref{rule:constant}, polynomial rule of \ref{rule:polynomial} and logarithm rule of \ref{rule:logarithm} that 
\[
\Lambda^\tau(\bb x^{k+1},\bb x^k) &:= \begin{pmatrix} \frac{\Delta}{\Delta x}V(\bb x^{k+v_{\sigma^{-1}(1)}}) &  \frac{\Delta}{\Delta y} V(\bb x^{k+v_{\sigma^{-1}(2)}})\end{pmatrix}\\
& = 
\begin{pmatrix} \frac{\Delta}{\Delta x}(\gamma \log x-\delta x) &  \frac{\Delta}{\Delta y}(\alpha \log y- \beta y)\end{pmatrix} \\
&= \begin{pmatrix} \gamma \dfrac{\log x_i^{k+1}-\log x_i^k}{x_i^{k+1}-x_i^k}-\delta &  \alpha \dfrac{\log y_i^{k+1}-\log y_i^k}{y_i^{k+1}-y_i^k}- \beta\end{pmatrix}.
\]
While we could proceed as before and use \eqref{discFEqn} to compute $\bb f^\tau$, here we illustrate a different approach using \eqref{discCond2}.
Since $V$ is independent of $t$, \eqref{discCond2} reduces to finding $\bb f^\tau$ belonging to the kernel of $\Lambda^\tau$ such that $\bb f^\tau$ is consistent to $\bb f$.
Since $\Lambda^\tau$ is only a $1\times 2$ matrix, $\bb f^\tau$ must be of the form
\[
\bb f^\tau(\bb x^{k+1},\bb x^k) = C^\tau(\bb x^{k+1},\bb x^k) \begin{pmatrix} \alpha \dfrac{\log y_i^{k+1}-\log y_i^k}{y_i^{k+1}-y_i^k}- \beta \\ \delta-\gamma \dfrac{\log x_i^{k+1}-\log x_i^k}{x_i^{k+1}-x_i^k}\end{pmatrix},
\] for some scalar $C^\tau$ to be determined. As $\bb f^\tau$ needs to be consistent to $\bb f$, taking the limit of $\bb f^\tau$ as $\tau\rightarrow 0$ implies $\displaystyle \lim_{\tau\rightarrow 0} C^\tau = x y$. In other words, $C^\tau := (xy)^\tau$ can be any consistent discretization of $xy$. For brevity, choosing $C^\tau := x^k y^k$ gives the conservative discretization of \eqref{LVsys2D},
\begin{equation*}
\bb F^\tau(\bb x^{k+1}, \bb x^k) := \begin{pmatrix}
\dfrac{x^{k+1}-x^{k}}{\tau} - x^k\left(\alpha y^k\left(\dfrac{\log y^{k+1} - \log y^k}{y^{k+1}-y^{k}}\right) - \beta y^k\right)\\
\dfrac{y^{k+1}-y^{k}}{\tau} - y^k\left(\delta x^k -\gamma x^k\left(\dfrac{\log x^{k+1} - \log x^k}{x^{k+1}-x^k}\right)\right)
\end{pmatrix} =\bb 0,
\label{LV2dDisc}
\end{equation*} which conserves $V(x,y)$.

\subsubsection{A degenerate 3-species system} \text{ }\\
Consider the Lotka--Volterra system with positive populations $x_1, x_2, x_3$ of 3-species,
\begin{align}
\bb F(\bb x) := \begin{pmatrix}\dot x_1-x_1(x_2-x_3) \\ \dot x_2-x_2(x_3-x_1) \\ \dot x_3-x_3(x_1-x_2) \end{pmatrix}=\boldsymbol 0. \label{LVsys3D}
\end{align}
From \cite{schi03a}, \eqref{LVsys3D} satisfies a degeneracy condition and thus has two conserved quantities
\[
\bb\psi(\bb x):=\begin{pmatrix}  x_1+x_2+x_3 \\ x_1x_2x_3 \end{pmatrix},
\] with the corresponding multiplier given by \eqref{multCond1},
\[
\Lambda(\bb x) := \begin{pmatrix} 1 & 1 & 1 \\ x_2x_3 & x_1x_3 & x_1x_2 \end{pmatrix}.
\]
Next we employ \eqref{discMultPerm} to discretize $\Lambda$. However, unlike the rigid body example and the 2-species Lotka--Volterra example, the discrete multiplier will in general depend on the permutation $\sigma\in S_3$ because the term $x_1x_2x_3$ of $\bb \psi$ depends explicitly on all three variables. Since there are $3!=6$ choices of $\sigma$, let us first look at the identity permutation $\sigma_1 = (1,2,3)$. For $\sigma_1$, $\bb v_1=(0,0,0)$, $\bb v_2=\bb E_1=(1,0,0)$ and $\bb v_3=\bb E_2=(1,1,0)$. So by \eqref{discMultPerm} and by the linearity rule of \ref{rule:linearity}, constant rule of \ref{rule:constant}, separable product rule of \ref{rule:sepProd} and polynomial rule of \ref{rule:polynomial},
\[
\Lambda^\tau(\bb x^{k+1}, \bb x^k)&=\begin{pmatrix}
\frac{\Delta}{\Delta x_1}\bb \psi(\bb x^{k+v_{\sigma_1^{-1}(1)}}) & \frac{\Delta}{\Delta x_2}\bb \psi(\bb x^{k+v_{\sigma_1^{-1}(2)}}) &\frac{\Delta}{\Delta x_2} \bb \psi(\bb x^{k+v_{\sigma_1^{-1}(3)}})
\end{pmatrix}\\
&=\begin{pmatrix}
\frac{\Delta}{\Delta x_1}\bb \psi(\bb x^{k+v_{1}}) & \frac{\Delta}{\Delta x_2}\bb \psi(\bb x^{k+v_{2}}) &\frac{\Delta}{\Delta x_2} \bb \psi(\bb x^{k+v_{3}})
\end{pmatrix}\\
&=\begin{pmatrix} 
\frac{\Delta}{\Delta x_1}\left(x_1+x_2+x_3\right) & \frac{\Delta}{\Delta x_2}\left(x_1+x_2+x_3\right) & \frac{\Delta}{\Delta x_3}\left(x_1+x_2+x_3\right)\\ 
x_2^k x_3^k \frac{\Delta}{\Delta x_1}x_1 & x_1^{k+1}x_3^k\frac{\Delta}{\Delta x_2}x_2 & x_1^{k+1}x_2^{k+1}\frac{\Delta}{\Delta x_3}x_3
\end{pmatrix}\\
&=\begin{pmatrix}
1 & 1 & 1\\
x_2^kx_3^k & x_1^{k+1}x_3^k & x_1^{k+1}x_2^{k+1}\\
\end{pmatrix}.
\]
The leftmost $2\times 2$ minor of $\Lambda^\tau$ is invertible for sufficiently small $\tau$ on \\$U=\{\bb x\in\mathbb{R}^3:x_1\ne x_2 \text{ and } x_3\ne 0\}$ and therefore
\[
[\tilde{\Lambda}^\tau(\bb x^{k+1},\bb x^k)]^{-1} &= \frac{1}{x_3^k(x_1^{k+1}-x_2^k)}
\begin{pmatrix}
x_1^{k+1}x_3^k & -1 \\ -x_2^kx_3^k & 1
\end{pmatrix},\quad \Sigma^\tau(\bb x^{k+1},\bb x^k)=\begin{pmatrix}
 1\\ x_1^{k+1}x_2^{k+1}
\end{pmatrix}.
\]
Again, thanks to~\eqref{discFEqn} for any $g^\tau(\bb x^{k+1},\bb x^k)$ that is consistent of first order to \\$g(\bb x) = x_3(x_1-x_2)$, we have
\[
\tilde{\bb f}^\tau(\bb x^{k+1},\bb x^k) &= -[\tilde{\Lambda}^\tau]^{-1}\Sigma^\tau g^\tau = \dfrac{1}{x_3^k(x_1^{k+1}-x_2^k)}\begin{pmatrix}
x_1^{k+1}(x_2^{k+1}-x_3^k)\\ x_2^kx_3^k-x_1^{k+1}x_2^{k+1}
 \end{pmatrix} g^\tau(\bb x^{k+1},\bb x^k).
\]
To simplify $\tilde{\bb f}^\tau$, it is natural to choose $\bb g^\tau$ as
\[
\bb g^\tau(\bb x^{k+1},\bb x^k) = x_3^k(x_1^{k+1}-x_2^k), \hskip 2mm \text{ leading to } \hskip 2mm 
\tilde{\bb f}^\tau(\bb x^{k+1},\bb x^k) = \begin{pmatrix}
x_1^{k+1}(x_2^{k+1}-x_3^k)\\ x_2^kx_3^k-x_1^{k+1}x_2^{k+1}
 \end{pmatrix}.
\]
Thus, the resulting conservative discretization of \eqref{LVsys3D} is
\begin{equation*}
\boldsymbol F_1^\tau(\bb x^{k+1},\bb x^k) := 
\begin{pmatrix}
\dfrac{x^{k+1}_1-x^k_1}{\tau} - x_1^{k+1}(x_2^{k+1}-x_3^k)\\
\dfrac{x^{k+1}_2-x^k_2}{\tau} - (x_2^kx_3^k-x_1^{k+1}x_2^{k+1})\\
\dfrac{x^{k+1}_3-x^k_3}{\tau} - x_3^k(x_1^{k+1}-x_2^k)
\end{pmatrix} = \boldsymbol 0, \text{ for }\sigma_1 = (1,2,3).
\end{equation*}
Similarly, one can carry out the systematic procedure to derive conservative schemes for the other five permutations. For completeness, we include them here.
\[\boldsymbol F_2^\tau(\bb x^{k+1},\bb x^k) &:= 
\begin{pmatrix}
\dfrac{x^{k+1}_1-x^k_1}{\tau} - x_1^{k+1}(x_2^{k}-x_3^{k+1})\\
\dfrac{x^{k+1}_2-x^k_2}{\tau} - x_2^k(x_3^k-x_1^{k+1})\\
\dfrac{x^{k+1}_3-x^k_3}{\tau} - (x_3^{k+1}x_1^{k+1}-x_2^kx_3^k)
\end{pmatrix} = \boldsymbol 0, \text{ for }\sigma_2 = (1,3,2).\\
\boldsymbol F_3^\tau(\bb x^{k+1},\bb x^k) &:= 
\begin{pmatrix}
\dfrac{x^{k+1}_1-x^k_1}{\tau} - (x_1^{k+1}x_2^{k+1}-x_1^kx_3^k)\\
\dfrac{x^{k+1}_2-x^k_2}{\tau} - x_2^{k+1}(x_3^k-x_1^{k+1})\\
\dfrac{x^{k+1}_3-x^k_3}{\tau} - x_3^k(x_1^{k}-x_2^{k+1})
\end{pmatrix} = \boldsymbol 0, \text{ for }\sigma_3 = (2,1,3).\\
\boldsymbol F_4^\tau(\bb x^{k+1},\bb x^k) &:= 
\begin{pmatrix}
\dfrac{x^{k+1}_1-x^k_1}{\tau} - x_1^{k}(x_2^{k+1}-x_3^k)\\
\dfrac{x^{k+1}_2-x^k_2}{\tau} - x_2^{k+1}(x_3^{k+1}-x_1^k)\\
\dfrac{x^{k+1}_3-x^k_3}{\tau} - (x_1^k x_3^k-x_2^{k+1}x_3^{k+1})
\end{pmatrix} = \boldsymbol 0, \text{ for }\sigma_4 = (2,3,1).\\
\boldsymbol F_5^\tau(\bb x^{k+1},\bb x^k) &:= 
\begin{pmatrix}
\dfrac{x^{k+1}_1-x^k_1}{\tau} - (x_1^kx_2^k-x_1^{k+1}x_3^{k+1})\\
\dfrac{x^{k+1}_2-x^k_2}{\tau} - x_2^k(x_3^{k+1}-x_1^k)\\
\dfrac{x^{k+1}_3-x^k_3}{\tau} - x_3^{k+1}(x_1^{k+1}-x_2^k)
\end{pmatrix} = \boldsymbol 0, \text{ for }\sigma_5 = (3,1,2).\\
\boldsymbol F_6^\tau(\bb x^{k+1},\bb x^k) &:= 
\begin{pmatrix}
\dfrac{x^{k+1}_1-x^k_1}{\tau} - x_1^{k}(x_2^{k}-x_3^{k+1})\\
\dfrac{x^{k+1}_2-x^k_2}{\tau} - (x_2^{k+1}x_3^{k+1}-x_1^{k}x_2^{k})\\
\dfrac{x^{k+1}_3-x^k_3}{\tau} - x_3^{k+1}(x_1^{k}-x_2^{k+1})
\end{pmatrix} = \boldsymbol 0, \text{ for }\sigma_6 = (3,2,1). 
\]
\subsection{The planar restricted three-body problem} \vskip-2mm\label{sec:3bodyDisc}\text{ }\\
Next we consider the planar restricted three-body problem described in \cite{Sze67}. This problem models the gravitational motion of three bodies in a plane with a negligible mass for one of the bodies, for example the Earth--Moon--Satellite system. The equations of motion can be expressed as a first order system:
\begin{equation}
\boldsymbol F(\boldsymbol x) 
=
\begin{pmatrix}\dot{x_1} - y_1 \\
\dot{x_2} -y_2\\
\dot{y_1} -\left(x_1+2y_2-\dfrac{\alpha(x_1-\beta)}{((x_1-\beta)^2+x_2^2)^{\frac{3}{2}}}-\dfrac{\beta(x_1+\alpha)}{((x_1+\alpha)^2+x_2^2)^{\frac{3}{2}}}\right)\\
\dot{y_2} -\left(x_2-2y_1-\dfrac{\alpha x_2}{((x_1-\beta)^2+x_2^2)^{\frac{3}{2}}}-\dfrac{\beta x_2}{((x_1+\alpha)^2+x_2^2)^{\frac{3}{2}}}\right)\end{pmatrix}=\boldsymbol{0}, \label{3bodySys}
\end{equation}
where $(x_1,x_2)$ is the position of the satellite relative to the center of mass of the Earth and Moon, $\alpha, \beta$ are relative masses of two bodies such that $\alpha+\beta=1$. It is well-known that  \eqref{3bodySys} has a conserved quantity called Jacobi integral $J$ given by
\[
 J(\boldsymbol x)=\dfrac{x_1^2 + x_2^2- y_1^2-y_2^2}{2} + \dfrac{\alpha}{((x_1-\beta)^2+x_2^2)^{\frac{1}{2}}}+\dfrac{\beta}{((x_1+\alpha)^2+x_2^2)^{\frac{1}{2}}}.
\]Moreover, there exists a canonical transformation which turns \eqref{3bodySys} into a Hamiltonian system with $J$ being the effective Hamiltonian in the new coordinates \cite{Sze67}. We shall work directly with \eqref{3bodySys} to illustrate the application of the multiplier method without the need to make or know the existence of such transformation. Using \eqref{multCond1}, the transpose of the associated $1\times 4$ multiplier matrix $\Lambda$ is given by
\[
\Lambda(\bb x)^T = \begin{pmatrix} x_1-\dfrac{\alpha(x_1-\beta)}{((x_1-\beta)^2+x_2^2)^{\frac{3}{2}}}-\dfrac{\beta(x_1+\alpha)}{((x_1+\alpha)^2+x_2^2)^{\frac{3}{2}}} \\
 x_2-\dfrac{\alpha x_2}{((x_1-\beta)^2+x_2^2)^{\frac{3}{2}}}-\dfrac{\beta x_2}{((x_1+\alpha)^2+x_2^2)^{\frac{3}{2}}} \\
 -y_1 \\
 -y_2
\end{pmatrix}.
\]
Note that $J$ is a linear combination of single variable functions of $x_i^2, y_i^2$ and two variable functions of $x_1, x_2$. Specifically, for any permutation $\sigma \in S_4$, $\bb x^{k+v_{\sigma^{-1}(1)}} = (x_1^k, x_2^s, *, *)$ where $s=k,k+1$ and $\bb x^{k+v_{\sigma^{-1}(2)}} = (x_1^r, x_2^k, *, *)$ where $r=k,k+1$. Thus, the discrete multiplier matrix simplifies to
\[
&\Lambda^\tau(\bb x^{k+1}, \bb x^k) =\\
& 
\begin{pmatrix} 
\frac{\Delta}{\Delta x_1} J(\bb x^{k+v_{\sigma^{-1}(1)}}) \\
\frac{\Delta}{\Delta x_2} J(\bb x^{k+v_{\sigma^{-1}(2)}}) \\
\frac{\Delta}{\Delta y_1} J(\bb x^{k+v_{\sigma^{-1}(3)}}) \\
\frac{\Delta}{\Delta y_2} J(\bb x^{k+v_{\sigma^{-1}(4)}})
 \end{pmatrix} =
 \begin{pmatrix} 
\overline{x}_1+\frac{\Delta}{\Delta x_1} \left(\dfrac{\alpha}{((x_1-\beta)^2+(x_2^s)^2)^{\frac{1}{2}}}+\frac{\beta}{((x_1+\alpha)^2+(x_2^s)^2)^{\frac{1}{2}}} \right)\\
\overline{x}_2+\frac{\Delta}{\Delta x_2} \left(\dfrac{\alpha}{((x_1^r-\beta)^2+x_2^2)^{\frac{1}{2}}}+\frac{\beta}{((x_1^r+\alpha)^2+x_2^2)^{\frac{1}{2}}}\right)\\
-\overline{y}_1 \\
-\overline{y}_2
 \end{pmatrix}
\]
By the reciprocal rule of \ref{rule:reciprocal}, chain rule of \ref{rule:chainScalar}, and rational power rule of \ref{rule:rationalPower}, 
\[
\frac{\Delta}{\Delta z}\left(\frac{1}{\sqrt{z}}\right) = -\frac{1}{\sqrt{z^k}\sqrt{z^{k+1}}(\sqrt{z^{k}}+\sqrt{z^{k+1}})}
\]
Combining with the chain rule of \ref{rule:chainScalar} and polynomial rule of \ref{rule:polynomial} gives
\begin{align*}
\frac{\Delta}{\Delta x_1}\left(\frac{1}{((x_1+\alpha)^2+(x_2^s)^2)^{\frac{1}{2}}}\right) &= -\frac{2(\overline{x}_1+\alpha)}{A^{k,s}A^{k+1,s}(A^{k,s}+A^{k+1,s})},\\
\frac{\Delta}{\Delta x_2}\left(\frac{1}{((x_1^r+\alpha)^2+x_2^2)^{\frac{1}{2}}}\right) &= -\frac{2\overline{x}_2}{A^{r,k}A^{r,k+1}(A^{r,k}+A^{r,k+1})}, \\
\frac{\Delta}{\Delta x_1}\left(\frac{1}{((x_1-\beta)^2+(x_2^s)^2)^{\frac{1}{2}}}\right) &= -\frac{2(\overline{x}_1-\beta)}{B^{k,s}B^{k+1,s}(B^{k,s}+B^{k+1,s})}, \\
\frac{\Delta}{\Delta x_2}\left(\frac{1}{((x_1^r-\beta)^2+x_2^2)^{\frac{1}{2}}}\right) &= -\frac{2\overline{x}_2}{B^{r,k}B^{r,k+1}(B^{r,k}+B^{r,k+1})},
\end{align*} where $A^{r,s} := \sqrt{(x_1^{r}+\alpha)^2+(x_2^{s})^2)}, B^{r,s} := \sqrt{(x_1^{r}-\beta)^2+(x_2^{s})^2}$. 
So, $\Lambda^\tau$ simplifies to
\[
\Lambda^\tau(\bb x^{k+1},\bb x^k)^T &= 
\begin{pmatrix} 
\overline{x}_1-\dfrac{2\alpha(\overline{x}_1-\beta)}{B^{k,s}B^{k+1,s}(B^{k,s}+B^{k+1,s})}-\dfrac{2\beta(\overline{x}_1+\alpha)}{A^{k,s}A^{k+1,s}(A^{k,s}+A^{k+1,s})} \\
\overline{x}_2-\dfrac{2\alpha\overline{x}_2}{B^{r,k}B^{r,k+1}(B^{r,k}+B^{r,k+1})}-\dfrac{2\beta\overline{x}_2}{A^{r,k}A^{r,k+1}(A^{r,k}+A^{r,k+1})}\\
-\overline{y}_1 \\
-\overline{y}_2
 \end{pmatrix}.
 \]
To find $\bb f^\tau$, we take the alternate approach of using \eqref{discCond2} as in the 2-species Lotka--Volterra example. Specifically, since $J$ is independent of $t$, we wish to find an approximation $\bb f^\tau$ belonging in the kernel of $\Lambda^\tau$ such that $\bb f^\tau$ is consistent to $\bb f$. Due to the similarity between the terms of $\Lambda^\tau$ and the form of $\bb f$, let us propose
\begin{align}
&\bb f^\tau(\bb x^{k+1},\bb x^k) := \label{R3B_RHS} \\
&C^\tau \begin{pmatrix} 
\overline{y}_1 \\
\overline{y}_2 \\
\overline{x}_1+2\overline{y}_2-\dfrac{2\alpha(\overline{x}_1-\beta)}{B^{k,s}B^{k+1,s}(B^{k,s}+B^{k+1,s})}-\dfrac{2\beta(\overline{x}_1+\alpha)}{A^{k,s}A^{k+1,s}(A^{k,s}+A^{k+1,s})}\\
\overline{x}_2-2\overline{y}_1-\dfrac{2\alpha\overline{x}_2}{B^{r,k}B^{r,k+1}(B^{r,k}+B^{r,k+1})}-\dfrac{2\beta\overline{x}_2}{A^{r,k}A^{r,k+1}(A^{r,k}+A^{r,k+1})}
 \end{pmatrix}, \nonumber
\end{align} for some $C^\tau$ to be determined. Indeed, one can verify that $\bb f^\tau\rightarrow 
\displaystyle (\lim_{\tau\rightarrow 0} C^\tau) \bb f$ as $\tau\rightarrow 0$ which implies $\displaystyle \lim_{\tau\rightarrow 0} C^\tau=1$. For simplicity, we pick $C^\tau=1$. Moreover, it can be checked that $\Lambda^\tau \bb f^\tau =\bb 0$. In other words, $\bb f^\tau$ satisfies  \eqref{discCond2} which implies four conservative discretizations for $J(x_1,x_2,y_1,y_2)$ given by \eqref{discFormula} with $D_t^\tau \bb x$ defined by \eqref{1stOrdDisc2} and $\bb f^\tau$ defined by \eqref{R3B_RHS} for $r,s \in \{k,k+1\}$.

\subsection{Damped Harmonic Oscillator}\label{sec:dhoDisc}\text{ }\\
Recall from Example \ref{dhoSys} of Section \ref{exCLMult}, the damped harmonic oscillator of \eqref{dhoSysEqn} has a time-dependent conserved quantity $\psi$ and $1\times 2$ multiplier $\Lambda$ given by,
\[
\psi(t,x,y) = \frac{e^{\frac{\gamma}{m}t}}{2}\left(my^2+\gamma x y + \kappa x^2\right),& \hskip 2mm 
\Lambda(t,x,y) = e^{\frac{\gamma}{m}t} \begin{pmatrix} \left(\kappa x+\dfrac{\gamma}{2}y\right)&\left(\dfrac{\gamma}{2}x+my\right)\end{pmatrix}.
\]
Similar to the 3-species Lotka--Volterra example, $\psi$ is a linear combination of functions with explicit dependence on $t,x,y$. Thus to employ \eqref{discMultPerm}, we need to choose a permutation $\sigma\in S_3$. For simplicity, we take $\sigma$ to be the identity permutation, which leads to the sequence $\bb v_0 =(0, 0, 0), \bb v_1 = \bb E_0 = (1, 0, 0)$, $\bb v_2 = \bb E_1 = (1, 1, 0)$. It follows from the linearity rule of \ref{rule:linearity}, constant rule of \ref{rule:constant}, polynomial rule of \ref{rule:polynomial} and exponential rule of \ref{rule:exponential},
\[
\Lambda^\tau &= \begin{pmatrix} \frac{\Delta}{\Delta x}\psi(\bb X^{\bb k+\bb v_{\sigma^{-1}(1)}}) & \frac{\Delta}{\Delta y}\psi(\bb X^{\bb k+\bb v_{\sigma^{-1}(2)}}) \end{pmatrix} =  \begin{pmatrix} \frac{\Delta}{\Delta x}\psi(\bb X^{\bb k+\bb v_1}) & \frac{\Delta}{\Delta y}\psi(\bb X^{\bb k+\bb v_2}) \end{pmatrix} \\
&= \begin{pmatrix} \frac{\Delta}{\Delta x}\psi(t^{k+1},x^k, y^k) & \frac{\Delta}{\Delta y}\psi(t^{k+1}, x^{k+1}, y^k) \end{pmatrix} \\
&= \frac{e^{\frac{\gamma}{m}t^{k+1}}}{2} \begin{pmatrix} \frac{\Delta}{\Delta x}(\gamma y^{k} x + \kappa x^2) & \frac{\Delta}{\Delta y}(m y^2 + \gamma x^{k+1} y) \end{pmatrix}  \\
&= e^{\frac{\gamma}{m}t^{k+1}} \begin{pmatrix} \kappa\bar{x}+\frac{\gamma}{2} y^k & \frac{\gamma}{2} x^{k+1}+m\bar{y} \end{pmatrix},\\
\partial_t^\tau \psi &= \frac{\Delta}{\Delta t}\psi(\bb X^{\bb k+\bb v_{\sigma^{-1}(0)}}) = \frac{\Delta}{\Delta t}\psi(\bb X^{\bb k+\bb v_{0}}) = \frac{\Delta}{\Delta t}\psi(\bb X^{\bb k})  \\
&= \frac{\Delta}{\Delta t}\psi(t^k, x^k, y^k) = e^{\frac{\gamma}{m}t^k} \left(\frac{e^{\frac{\gamma}{m}\tau}-1}{\frac{\gamma}{m}\tau}\right)\frac{\gamma}{2m}\left(m(y^k)^2+\gamma x^k y^k + \kappa (x^k)^2\right).
\]
To find $\bb f^\tau$ satisfying \eqref{discCond2}, let us use the alternate approach presented before in the 2-species Lotka--Volterra system and the planar restricted three-body problem. Noting the form of $\bb f$, we define
\[
\bb f^\tau(\bb X^{\bb k+1},\bb X^{\bb k}):=C^\tau(\bb X^{\bb k+1},\bb X^{\bb k})\begin{pmatrix}
 \tilde{y}^\tau \\
-\frac{1}{m} \left(\gamma y^\tau+\kappa x^\tau\right)
\end{pmatrix},
\] where the undetermined consistent terms satisfy $C^\tau \rightarrow 1$, $x^\tau\rightarrow x$ and $\tilde{y}^\tau, y^\tau\rightarrow y$, as $\tau\rightarrow 0$. For this choice of $\bb f^\tau$, treating \eqref{discCond2} as a constraint implies
\begin{align}
&C^\tau e^{\frac{\gamma}{m}t^{k+1}}\left(\kappa(\overline{x}\tilde{y}^\tau-\overline{y}x^\tau)+\gamma\left(\frac{y^k \tilde{y}^\tau}{2}-\overline{y} y^\tau\right) 
-\frac{\gamma^2}{2m}x^{k+1}y^\tau-\frac{\gamma \kappa}{2m}x^{k+1}x^\tau\right) \label{dhoCond}\\
&\hskip 2mm= - e^{\frac{\gamma}{m}t^{k}}\left(\frac{e^{\frac{\gamma}{m}\tau}-1}{\frac{\gamma}{m}\tau}\right) \left(\frac{\gamma}{2}(y^k)^2+\frac{\gamma^2}{2m}x^ky^k+\frac{\gamma \kappa}{2m}(x^k)^2\right). \nonumber
\end{align}
To simplify \eqref{dhoCond}, we choose $x^\tau:=\bar{x}$, $\tilde{y}^\tau:= \bar{y}$ and,
\[
C^\tau(\bb X^{\bb k+1},\bb X^{\bb k}) := \left(\frac{1-e^{-\frac{\gamma}{m}\tau}}{\frac{\gamma}{m}\tau}\right) \rightarrow 1 \text{ as } \tau\rightarrow 0.
\] Thus, substituting these choices for $x^\tau, \tilde{y}^\tau, C^\tau$ in \eqref{dhoCond} implies $y^\tau$ satisfies,
\[
y^\tau:=\dfrac{y^k\left(m\frac{y^k+\overline{y}}{2}+\frac{\gamma}{2}x^k\right)+\frac{\kappa}{2}\left((x^k)^2-x^{k+1}\overline{x}\right)}{m\overline{y}+\frac{\gamma}{2}x^{k+1}} \rightarrow y \text{ as } \tau\rightarrow 0.
\]
Finally, we have the conservative discretization of the damped harmonic oscillator,
\begin{align*}
&\bb F^\tau(\bb X^{\bb k+1},\bb X^{\bb k}) := \\
&\begin{pmatrix} \frac{x^{k+1}-x^k}{\tau} \\  \frac{y^{k+1}-y^k}{\tau}  \end{pmatrix}-
\left(\frac{1-e^{-\frac{\gamma}{m}\tau}}{\frac{\gamma}{m}\tau}\right) \begin{pmatrix} 
\bar{y} \\  
-\frac{1}{m}\left( \gamma \frac{y^k\left(m\frac{y^k+\overline{y}}{2}+\frac{\gamma}{2}x^k\right)+\frac{\kappa}{2}\left((x^k)^2-x^{k+1}\overline{x}\right)}{m\overline{y}+\frac{\gamma}{2}x^{k+1}}+\kappa\overline{x}\right)  \end{pmatrix}= \bb 0,
\end{align*}which conserves the time-dependent conserved quantity $\psi(t,x,y)$. Note that $\bb F^\tau$ reduces to the midpoint rule when $\gamma\rightarrow 0^+$ which is known to conserve the energy of the harmonic oscillator.

\section{Numerical results}\label{sec:numerics}\text{ }\\
In this section, we report numerical results verifying conservative properties of the discretizations derived in the examples of Section \ref{sec:CM_examples}. All discretizations derived in the examples are implicit and at least first order accurate. We compare results with the backward Euler, the trapezoidal method and the midpoint method. The last two methods are implicit and second order, with the midpoint method being also symplectic for Hamiltonian systems \cite{hair06Ay}. A fixed point iteration was used to solve non-linear systems with an absolute tolerance $TOL=10^{-15}$ for all implicit methods unless otherwise noted. All numerical results begin at $t=0$ and end at a final time $T$. We have used a uniform time step of size $\tau$ with a total number of $N$ time steps. The error in the component $\psi_i$ of a conserved quantity $\bb \psi$ is measured by: 
\[
\text{Error}[\psi_i(t,\bb x)] := \max_{k=1, \dots, N} |\psi_i(t^k,\bb x^k)-\psi_i(0,\bb x^0)|
\]

\subsection{Euler's equation for rigid body rotation}\text{ }\\
For the first example, $T=10$, $N=1000$ and $\tau=0.01$ with parameters $I_{1}=1$,
$I_{2}=2$, $I_{3}=3$, and initial conditions $\bb \omega(0) = (1, 1, 1)^T$.
\begin{table}[H]
\noindent \begin{centering}
\begin{tabular}{|c|c|c|}
\hline 
Method & $\text{Error}\left[E(\bb \omega)\right]$ & $\text{Error}\left[L(\bb \omega)\right]$\tabularnewline
\hline 
\hline 
Backward Euler  & $2.71\cdot10^{-2}$ & $6.18\cdot10^{-2}$\tabularnewline
\hline 
Multiplier - same as Midpoint  & $3.997\cdot10^{-15}$ & $3.997\cdot10^{-15}$\tabularnewline
\hline 
Trapezoidal  & $5.09\cdot10^{-6}$ & $8.33\cdot10^{-6}$\tabularnewline
\hline 
\end{tabular}
\par\end{centering}
\caption{Numerical error of $E$ and $L$ for the 3D rigid body rotation example.}\vskip -5mm
\end{table}
The multiplier method guarantees conservation of energy $E$ and angular momentum $L$ up to round-off errors. Interestingly, the multiplier method reduces to the midpoint rule
for this problem.

\subsection{2-species Lotka--Volterra system}\text{ }\\
For the second example, $T=10$, $N=1000$ and $\tau=0.01$ with parameters $\alpha=\beta=\gamma=\delta=1$, and initial conditions $\bb x(0) = (1, 1)^T$. Additionally, the tolerance for the fixed point iteration was set to $TOL=10^{-13}$ due to small divisors approaching machine precision\footnote{One can avoid this by a Taylor expansion since the numerator of the divided difference will also be small. Here, we chose not to make this regularization and compare numerical results as is.} which can arise from divided difference quotients in the multiplier discretization. 

\begin{table}[H]
\noindent \begin{centering}
\begin{tabular}{|c|c|}
\hline 
Method & $\text{Error}\left[V(\bb x)\right]$\tabularnewline
\hline 
\hline 
Backward Euler & $2.71\cdot10^{-2}$\tabularnewline
\hline 
Multiplier & $1.11\cdot10^{-14}$\tabularnewline
\hline 
Midpoint & $7.32\cdot10^{-6}$\tabularnewline
\hline 
Trapezoidal  & $1.46\cdot10^{-5}$\tabularnewline
\hline 
\end{tabular}
\par\end{centering}
\protect\caption{Numerical error of $V$ for the 2-species Lotka--Volterra example.}
\vskip -5mm
\end{table}
The multiplier method is the only method from those tested which guarantees
conservation of $V$ up to round-off errors.

\subsection{3-species Lotka--Volterra system}\text{ }\\
For the third example, $T=10$, $N=1000$ and $\tau=0.01$ with initial conditions $\bb x=(1,2,3)^T$.

\begin{table}[H]
\noindent \begin{centering}
\begin{tabular}{|c|c|c|}
\hline 
Method & $\text{Error}\left[x+y+z\right]$ & $\text{Error}\left[xyz\right]$\tabularnewline
\hline 
\hline 
Backward Euler & $1.07\cdot10^{-14}$ & $1.299\cdot10^{0}$\tabularnewline
\hline 
Multiplier - $\bb F_{1}^{\tau}$ discretization & $5.33\cdot10^{-15}$ & $1.42\cdot10^{-14}$\tabularnewline
\hline 
Midpoint & $6.22\cdot10^{-15}$ & $4.17\cdot10^{-5}$\tabularnewline
\hline 
Trapezoidal  & $7.99\cdot10^{-15}$ & $8.34\cdot10^{-5}$\tabularnewline
\hline 
\end{tabular}
\par\end{centering}

\protect\caption{Numerical error of $x+y+z$ and $xyz$ for the 3-species Lotka--Volterra example.}\vskip -5mm
\end{table}
The multiplier method is the only method of those tested which guarantees conservation of both conserved quantities up to round-off errors. Additionally, we compared the six discretizations generated by the permutations of $\sigma\in S_3$ and verified that every discretization conserves both conserved quantities up to round-off errors as expected.

\begin{table}[H]
\noindent \begin{centering}
\begin{tabular}{|c|c|c|}
\hline 
Method & $\text{Error}\left[x+y+z\right]$ & $\text{Error}\left[xyz\right]$\tabularnewline
\hline 
\hline 
Multiplier - $\bb F_{1}^{\tau}$ discretization & $5.33\cdot10^{-15}$ & $1.42\cdot10^{-14}$\tabularnewline
\hline 
Multiplier - $\bb F_{2}^{\tau}$ discretization & $7.11\cdot10^{-15}$ & $1.33\cdot10^{-14}$\tabularnewline
\hline 
Multiplier - $\bb F_{3}^{\tau}$ discretization & $3.55\cdot10^{-15}$ & $1.24\cdot10^{-14}$\tabularnewline
\hline 
Multiplier - $\bb F_{4}^{\tau}$ discretization & $7.11\cdot10^{-15}$ & $1.33\cdot10^{-14}$\tabularnewline
\hline 
Multiplier - $\bb F_{5}^{\tau}$ discretization & $5.33\cdot10^{-15}$ & $1.24\cdot10^{-14}$\tabularnewline
\hline 
Multiplier - $\bb F_{6}^{\tau}$ discretization & $5.33\cdot10^{-15}$ & $1.78\cdot10^{-14}$\tabularnewline
\hline 
\end{tabular}
\par\end{centering}

\protect\caption{Comparison of all conservative discretizations generated by permutations
of $\sigma \in S_3$.}\vskip -5mm
\end{table}

\subsection{Planar restricted three-body problem}\text{ }\\
For the fourth example, we have used the standard Arenstorf orbit parameters. Specifically, $T=17.0652165601579625588917206249$, $N=200000$, $\tau\approx8.5326\cdot10^{-5}$ with parameters $\alpha=0.012277471$, $\beta=1-\alpha$ and initial conditions
\[
\begin{pmatrix}x_{1}(0)\\
x_{2}(0)\\
y_{1}(0)\\
y_{2}(0)
\end{pmatrix}=\begin{pmatrix}0.994\\
0\\
0\\
-2.00158510637908252240537862224
\end{pmatrix}.
\]

\begin{table}[H]
\noindent \begin{centering}
\begin{tabular}{|c||c|}
\hline 
Method & $\text{Error}\left[J(\bb x)\right]$\tabularnewline
\hline 
\hline 
Backward Euler & $3.22\cdot10^{-2}$\tabularnewline
\hline 
\hline 
Multiplier & $8.10\cdot10^{-14}$\tabularnewline
\hline 
\hline 
Midpoint & $2.48\cdot10^{-4}$\tabularnewline
\hline 
\hline 
Trapezoidal & $1.82\cdot10^{-4}$\tabularnewline
\hline 
\end{tabular}
\par\end{centering}

\protect\caption{Numerical error in $J$ for the planar restricted three-body problem
example.}\vskip -5mm
\end{table}
As in the previous examples, the multiplier method is the only method of those tested that guarantees conservation of the Jacobi integral $J$ up to round-off errors.

\subsection{Damped harmonic oscillator}\text{ }\\
For the last example, $T=10$, $N=1000$ and $\tau=0.01$ with parameters $m=4$, $\gamma=0.5$, $\kappa=5$, and initial conditions $\bb x(0) = (1, 0)^T$.

\begin{table}[H]
\noindent \begin{centering}
\begin{tabular}{|c|c|}
\hline 
Method & $\text{Error}\left[\psi\left(t,x,y\right)\right]$\tabularnewline
\hline 
\hline 
Backward Euler & $2.92\cdot10^{-1}$\tabularnewline
\hline 
Multiplier & $5.77\cdot10^{-14}$\tabularnewline
\hline 
Midpoint & $9.72\cdot10^{-5}$\tabularnewline
\hline 
Trapezoidal  & $9.72\cdot10^{-5}$\tabularnewline
\hline 
\end{tabular}
\par\end{centering}

\protect\caption{Numerical error in $\psi$ for the damped harmonic oscillator example.}\vskip -5mm
\end{table}
Again, the multiplier method is the only method of those tested which guarantees conservation of $\psi$ up to round-off errors.

\section{Conclusion}\text{ }\\
In this paper, we have further developed the framework of the multiplier method applied to quasilinear ODE systems, originally put forward in~\cite{wan15a} and based on ideas from ~\cite{wan16}. Specifically, we showed that conservative schemes can be derived systematically for general dynamical systems. In particular, the multiplier method can be directly applied to non-Hamiltonian and non-autonomous systems without any reformulation or transformations of the original system. The method is fully systematic and can in principle yield consistent conservative discretization schemes of arbitrary order, though here we have restricted ourselves to conservative schemes that are at least first order accurate. The systematic construction of higher order conservative schemes using the multiplier approach is currently being explored. Moreover, we are currently investigating conservative semi-discretizations of PDEs using the multiplier approach.

In the application of the multiplier method, one practical difficulty which can arise is the need to invert a~$m\times m$ minor~$\tilde{\Lambda}$ of the multiplier matrix, where~$m$ is the number of conserved quantities to be preserved. This inversion is generally feasible when the dynamical systems has only a small number of conserved quantities, as it is typically the case. Moreover, the explicit inversion can often be avoided using the consistency and the form of the right hand side~$\bb f$, as demonstrated through several examples in this paper.

We point out that the conservative schemes derived using the multiplier method are not unique. In principle, one can make arbitrary choices of consistent discretizations for the multiplier matrix and the right hand side components~$\bb g$, though technical difficulties, such as small divisors, can arise if they are not compatible choices. One systematic choice for the multiplier matrix is made in this paper by the use of divided difference calculus, leading to potentially $n!$ distinct conservative schemes for a system with $n$ variables, as shown in the 3-species Lotka--Volterra example. Thus, the multiplier method is particularly flexible and can potentially be combined with other geometric numerical integration methods. 

While the different examples presented here highlight the generality of the multiplier method, they are all relatively small dynamical systems. Using the multiplier method, we are currently investigating large dynamical systems including the $n$-body problem, the $n$-species Lotka--Volterra system and the $n$-point vortex problem in the plane and on the sphere. Their findings and numerical results will be presented elsewhere, and will confirm the general applicability of the multiplier method for large dynamical systems.
\appendix
\section{Local invertibility of discrete multiplier matrix}\label{sec:localInverse}
\begin{lemma}\label{localInverse}
Let $n,m\in\mathbb{N}$ with $1\leq m \leq n$. Suppose $\tilde{\Lambda}\in C(B_R(s)\times B_R(\bb y)\rightarrow M_{m\times m}(\mathbb{R}))$ is invertible on some closed balls $B_R(s)\times B_R(\bb y)\subset I\times U$ of $(s,\bb y)\in I\times U$. Define
\[ \epsilon:=\displaystyle \min_{(t,\bb x)\in B_R(s)\times B_R(\bb y)} |\det(\tilde{\Lambda}(t,\bb x))|>0.
\]
Also assume $\tilde{\Lambda}^\tau = \tilde{\Lambda}+\mathcal{O}(\tau^q)$ and $\{\tilde{\Lambda}^\tau\}_{0<\tau<\tau_0}$ is equicontinuous on $B_R(s)\times B_R(\bb y)\times \cdots \times B_R(\bb y)$. Then, there exist constants $r(s,\bb y,\epsilon)$ with $0<r\leq R$ and $\tau^*(s,\bb y,\epsilon,\tau_0)>0$ such that if $0<\tau< \tau^*$ and $(t,\bb x^{k+1},\dots, \bb x^{k-r+1}) \in B_r(s)\times B_r(\bb y)\times\cdots\times B_r(\bb y)$, the inverse of $\tilde{\Lambda}^\tau(t,\bb x^{k+1},\dots,\bb x^{k-r+1})$ exists and there is a uniform bound $C(s,\bb y,\epsilon)>0$ satisfying,
\[
\norm{[\tilde{\Lambda}^\tau(t,\bb x^{k+1},\dots,\bb x^{k-r+1})]^{-1}} \leq C<\infty.
\]
\end{lemma}

\begin{proof}Since the determinant function, $\det\colon M_{m\times m}(\mathbb{R})\rightarrow \mathbb{R}$, is continuous, there exist $\delta_1(s,\bb y,\epsilon)>0$ such that if $A\in M_{m\times m}(\mathbb{R})$ and $\norm{A-\Lambda(s,\bb y)}\leq \delta_1$,
\begin{equation}
\norm{\det(A)-\det(\Lambda(s,\bb y))}\leq \frac{\epsilon}{2}. \label{lem:det}
\end{equation}
Similarly, since the adjugate function, $adj\colon M_{m\times m}(\mathbb{R})\rightarrow M_{m\times m}(\mathbb{R})$, is continuous, there exists $\delta_2(s,\bb y)>0$ such that if $A\in M_{m\times m}(\mathbb{R})$ and $\norm{A-\Lambda(s,\bb y)}\leq \delta_2$,
\begin{equation}
\norm{adj(A)-adj(\Lambda(s,\bb y))}\leq 1. \label{lem:adj}
\end{equation}
Define $\delta(s,\bb y,\epsilon) := \min\{\delta_1(s,\bb y,\epsilon), \delta_2(s,\bb y)\}$. Then by equicontinuity of $\tilde{\Lambda}^\tau$ on \\$B_R(s)\times B_R(\bb y)\times\cdots B_R(\bb y)$, for all $0<\tau<\tau_0$ there exists a constant $r(s,\bb y, \delta)$ with $0<r\leq R$ such that if $(t,\bb x^{k+1},\dots, \bb x^{k-r+1}) \in B_r(s)\times B_r(\bb y)\times\cdots\times B_r(\bb y)$,
\begin{equation}
\norm{\tilde{\Lambda}^\tau(t,\bb x^{k+1},\dots, \bb x^{k-r+1})-\tilde{\Lambda}^\tau(s,\bb y,\dots, \bb y)}<\frac{\delta}{2}. \label{lem:equi}
\end{equation}
Let $\bb x(t)=\bb y$ be the constant function with $\norm{\bb x}_{C^p(I^k)}=\norm{\bb y}$ and define $\tau^*(\bb y,\delta,\tau_0):=\min\left\{\tau_0, \left(\frac{\delta}{2C_{\Lambda}\norm{\bb y}}\right)^{1/q}\right\}>0$. Then by consistency of $\tilde{\Lambda}^\tau$, if $0<\tau<\tau^*$ with $t^k=s$,
\begin{align}\label{lem:cons}
\norm{\tilde{\Lambda}^\tau(s,\bb y,\dots, \bb y)-\Lambda(s,\bb y)}&=\norm{\tilde{\Lambda}^\tau(t^k,\bb x(t^{k+1}),\dots, \bb x(t^{k-r+1}))-\Lambda(t^k,\bb x(t^k))}\\&
\leq C_\Lambda \norm{\bb y} \tau^q \leq \frac{\delta}{2}.\nonumber
\end{align}
Combining \eqref{lem:equi} and \eqref{lem:cons}, for any $(t,\bb x^{k+1},\dots, \bb x^{k-r+1}) \in B_r(s)\times B_r(\bb y)\times\cdots\times B_r(\bb y)$ and $0\leq\tau\leq \tau^*$,
\begin{align}\nonumber
\norm{\tilde{\Lambda}^\tau(t,\bb x^{k+1},\dots, \bb x^{k-r+1})-\Lambda(s,\bb y)}&\leq \norm{\tilde{\Lambda}^\tau(t,\bb x^{k+1},\dots, \bb x^{k-r+1})-\tilde{\Lambda}^\tau(s,\bb y,\dots, \bb y)}\\
& \hskip 2mm + \norm{\tilde{\Lambda}^\tau(s,\bb y,\dots, \bb y)-\Lambda(s,\bb y)} <\delta.\label{lem:ineq}
\end{align} And since $\delta\leq \delta_1$, \eqref{lem:det} implies
\begin{align}\label{lem:detDisc}
&\left|\det(\tilde{\Lambda}^\tau(t,\bb x^{k+1},\dots, \bb x^{k-r+1}))\right| \\
&\hskip 4mm\geq \underbrace{\left|\det(\Lambda(s,\bb y))\right|}_{\geq \epsilon} \underbrace{-\left|\det\tilde{\Lambda}^\tau(t,\bb x^{k+1},\dots, \bb x^{k-r+1})-\det\Lambda(s,\bb y)\right|}_{\geq -\epsilon/2} \geq \frac{\epsilon}{2}. \nonumber
\end{align}
In other words, $\tilde{\Lambda}^\tau$ is invertible for all $(t,\bb x^{k+1},\dots, \bb x^{k-r+1})\in B_r(s)\times B_r(\bb y)\times\cdots\times B_r(\bb y)$ and $0\leq\tau\leq \tau^*$. To show a uniform upper bound on the norm of $[\tilde{\Lambda}^\tau]^{-1}$, it follows from the adjugate formula, \eqref{lem:ineq}, \eqref{lem:detDisc} and \eqref{lem:adj} (since $\delta\leq \delta_2$) that,
\[
&\norm{[\tilde{\Lambda}^\tau(t,\bb x^{k+1},\dots, \bb x^{k-r+1})]^{-1}} \\
&\hskip 4mm \leq \frac{\left(\norm{adj(\Lambda(s,\bb y))}+\norm{adj(\tilde{\Lambda}^\tau(t,\bb x^{k+1},\dots, \bb x^{k-r+1}))-adj(\Lambda(s,\bb y))}\right)}{\left|\det(\tilde{\Lambda}^\tau(t,\bb x^{k+1},\dots, \bb x^{k-r+1}))\right|}\\
&\hskip 4mm \leq\frac{2}{\epsilon}\left(\norm{adj(\Lambda(s,\bb y))}+1\right)=:C(s,\bb y,\epsilon)<\infty.
\]
\end{proof}

\section{Divided difference calculus}
\label{sec:divDiff}
$\text{ }$\\
In this appendix, we construct discrete calculus rules for first order forward divided differences of multivariate functions in $t\in \mathbb{R}$ and $\bb x\in \mathbb{R}^n$. For simplicity, we consider $\bb x$ with only real entries, though analogous derivations can be carried with complex entries. Also, similar constructions can be derived for backward and centered divided differences.

Since we will be discussing approximations at different time steps, it is convenient to introduce the following multi-index notations. Let $\alpha \in \mathbb N^{n}$ be a multi-index. A vector $\bb x=(x_1,\dots,x_n)\in \mathbb{R}^n$ with components at different time steps $\alpha$ is denoted as $\bb x^\alpha:=(x_1^{\alpha_1},\dots,x_n^{\alpha_n})$. For convenience, when the context is clear, we use the shorthand $k=(k,\dots,k) \in \mathbb{N}^n$. For $i=1,\dots, n$, we write $e_i\in \mathbb{N}^n$ as the multi-index with $1$ in the $i$-th component and $0$ elsewhere.

In anticipation of explicit time-dependent functions, the following notations will be used throughout this article. Specifically, let $\alpha \in \mathbb N^{n}$ and $\alpha_0\in \mathbb{N}$. To distinguish the time component, we denote a ``space-time" multi-index as $\bb \alpha=(\alpha_0,\alpha)$ and denote a space-time vector $\bb X=(t,\bb x)\in \mathbb{R}^{n+1}$ at different time steps $\bb \alpha$ as $\bb X^{\bb \alpha}:=(t^{\alpha_0},\bb x^\alpha)$. Similarly, when the context is clear, we abbreviate $\bb k=(k,\dots, k) \in \mathbb{N}^{n+1}$. For $i=0,\dots, n$, denote $\bb e_i\in \mathbb{R}^{n+1}$ as the space-time multi-index with $1$ in the $i$-th component and $0$ elsewhere.

\subsection{Forward difference}
\text{ }\\
Let $\bb X^{\bb k} = (t^k, \bb x^{k})\in I\times U$ and $\bb f\in C(I\times U\rightarrow\mathbb{R}^n)$. 

\begin{definition}
The forward difference of $\bb f$ at $\bb X^{\bb k}$ is the linear operator $\Delta: C(I\times U\rightarrow \mathbb{R}^n)\rightarrow \mathbb{R}^n$,
\[
\Delta \bb f({\bb X}^{\bb k}) :=\bb f(\bb X^{\bb {k+1}}) -\bb f(\bb X^{\bb k}) =\bb f(t^{k+1}, \bb x^{k+1})-\bb f(t^{k}, \bb x^{k}).
\]
\end{definition}It follows from continuity of $\bb f$ that $\displaystyle \lim_{{\bb X}^{\bb {k+1}}\rightarrow{\bb X}^{\bb k}}\Delta \bb f({\bb X}^{\bb k}) = 0$.
\begin{definition}
For $i=0,\dots, n$, the $i$-th partial forward difference of $\bb f$ at $\bb X^{\bb k}$ is the linear operator $\Delta_i: C(I\times U\rightarrow\mathbb{R}^n)\rightarrow \mathbb{R}^n$,
\[
\Delta_i \bb f(\bb X^{\bb k}) := \bb f(\bb X^{\bb {k}+\bb e_i})-\bb f(\bb X^{\bb {k}}).
\] Denote the partial forward difference of $f$ with respect to $t$ at $\bb X^{\bb k}$ as
\[
\Delta_0 \bb f(\bb X^{\bb k}) := \bb f(t^{k+1}, \bb x^{k})-\bb f(t^{k}, \bb x^{k}),
\]
and denote the partial forward difference of $\bb f$ with respect to $x_i$ at $\bb X^{\bb k}$ as
\[
\Delta_i \bb f(\bb X^{\bb k})= \bb f(t^{k}, \bb x^{k+e_i})-\bb f(t^{k}, \bb x^{k}), \text{ for }i=1,\dots, n.
\]
\end{definition}
\begin{lemma}
For any fixed permutation $\sigma\in S_{n+1}$ of the finite set $\{0,\dots, n\}$, define the sequence of vectors $\bb{v}_{i+1}=\bb{v}_{i}+\bb e_{\sigma(i)}\in \mathbb{N}^{n+1}$ with $\bb v_0=\bb 0$. Then $\bb v_{n+1}=\bb 1$ and $\Delta \bb f=\Delta_\sigma f$ where,
\begin{equation}
\Delta_\sigma \bb  f(\bb X^{\bb k}) := \sum_{i=0}^n \Delta_{\sigma(i)} \bb f(\bb x^{\bb k+\bb{v}_i}) =  \sum_{i=0}^n \Delta_{i} \bb f(\bb x^{\bb k+\bb{v}_{\sigma^{-1}(i)}}), \hskip 3mm \bb f\in C(I\times U\rightarrow \mathbb{R}^n). \label{permFormula}
\end{equation}
\end{lemma}
\begin{proof} $\bb v_{n+1}=\bb 1$ follows from injectivity of $\sigma\in S_{n+1}$. Then combining with definition of $\bb v_i$,
\[
\Delta \bb f(\bb x^{\bb k}) &= \left(\bb f(\bb X^{\bb {k}+\bb v_{n+1}})- \bb f(\bb X^{\bb {k}+\bb v_{n}}) \right)+\dots+\left( \bb f(\bb X^{\bb {k}+\bb v_{1}})- \bb f(\bb X^{\bb k+\bb v_0})\right) \\
&= \left(\bb f(\bb X^{\bb {k}+\bb v_{n}+\bb e_{\sigma(n)}})- \bb f(\bb X^{\bb {k}+\bb v_{n}}) \right)+\dots+\left(\bb  f(\bb X^{\bb {k}+\bb v_{0}+\bb e_{\sigma(0)}})- \bb f(\bb X^{\bb k+\bb v_0})\right) \\
&= \sum_{i=0}^n \Delta_{\sigma(i)} \bb f(\bb X^{\bb k+\bb{v}_i}) = \Delta_\sigma \bb f(\bb X^{\bb k}).\]
\end{proof}
\begin{remark}If $\bb f$ is time independent, then $\sigma \in S_n$ of the finite set $\{1,\dots, n\}$ and the sequence $\bb{v}_{i+1}=\bb{v}_{i}+\bb e_{\sigma(i)}\in \mathbb{N}^{n}$ starts with $\bb v_1 = 0$.
\end{remark}
There are many equivalent forms of \eqref{permFormula}. Making the specific choices for $\sigma$ in \eqref{permFormula} with the identity permutation and the reversal permutation $(0 \dots n)\rightarrow (n \dots 0)$ gives the following two lexicographically-ordered forward differences.
\begin{definition} The (lexicographically-ordered) increasing and decreasing forw-ard differences of $\bb f$ at $\bb X^{\bb k}$ are the linear operators $\Delta_{inc}, \Delta_{dec}:C(I\times U\rightarrow \mathbb{R}^n)\rightarrow \mathbb{R}^n$,
\begin{align*}
\Delta_{inc} \bb f(\bb X^{\bb k}) &= \sum_{i=0}^n \Delta_{i} \bb f(\bb X^{\bb k+\bb{E}_i}), \hskip 3mm f\in C(I\times U\rightarrow \mathbb{R}),\\
\Delta_{dec} \bb f(\bb X^{\bb k}) &= \sum_{i=0}^n \Delta_{i} \bb f(\bb X^{\bb {k+1}-\bb{E}_i}), \hskip 3mm f\in C(I\times U\rightarrow \mathbb{R}).
\end{align*} where ${\bb E}_i\in\mathbb{R}^{n+1}$ with ones in the first $i$ components and zero in the last $n+1-i$ components and $\Delta_{inc} \bb f=\Delta \bb f=\Delta_{dec} \bb f$.
\end{definition}
Moreover, since there are $(n+1)!$ choices for $\sigma$, summing up the different $\sigma\in S_{n+1}$ gives a symmetrized form of \eqref{permFormula}. 
\begin{definition} The symmetrized forward difference of $\bb f$ at $\bb X^{\bb k}$ is the linear operator $\Delta_{sym}:C(I\times U\rightarrow \mathbb{R}^n)\rightarrow \mathbb{R}^n$,
\begin{align}
\Delta_{sym} \bb f(\bb X^{\bb k}) &= \frac{1}{(n+1)!} \sum_{i=0}^n\sum_{\sigma \in S_{n+1}} \Delta_{\sigma(i)} \bb f(\bb X^{\bb k+\bb{v}_i}) \label{symForDiv} \\
&= \sum_{i=0}^n \Delta_{i} \bb f_{sym}(\bb X^{\bb k}), \hskip 3mm f\in C(I\times U\rightarrow \mathbb{R}),\nonumber
\end{align}where $\bb f_{sym}(\bb X^{\bb k}) := \frac{1}{(n+1)!}\sum_{\sigma \in S_{n+1}}  \bb f(\bb X^{\bb k+\bb{v}_{\sigma(i)}})$ and $\Delta \bb f=\Delta_{sym} \bb f$.
\end{definition}

\subsection{Divided difference}
\text{ }\\
Next, we define a version of multivariate divided difference. For convenience, we use $x_0:=t$ interchangeably so that $\Delta x_0:=\Delta t := t^{k+1}-t^k$ and $\Delta x_i := \Delta x_i^k = x_i^{k+1}-x_i^k$ for $i=1,\dots,n$. 
\begin{definition}
Let $\bb f\in C^1(I\times U\rightarrow\mathbb{R}^n)$ and $\bb X^{\bb \alpha}\in I\times U$. For $i=0,\dots,n$, the $i$-th (first order) divided difference of $f$ at $\bb X^{\bb \alpha}$ is the linear operator $\frac{\Delta}{\Delta_i}:C^1(I\times U\rightarrow\mathbb{R}^n)\rightarrow \mathbb{R}^n$,
\[
\frac{\Delta}{\Delta x_i}\bb f(\bb X^{\bb \alpha}):=\frac{\Delta_i \bb f(\bb X^{\bb \alpha})}{\Delta x_i} = \frac{\bb f(\bb X^{\bb \alpha+\bb e_i})-\bb f(\bb X^{\bb \alpha})}{x_i^{k+1}-x_i^k}, \] 
\end{definition} where $\displaystyle \lim_{\bb X^{\bb \alpha+\bb e_i}\rightarrow \bb X^{\bb \alpha}} \frac{\Delta}{\Delta x_i}\bb f(\bb X^{\bb \alpha}) = \partial_{x_i}\bb f(t^{\alpha_0},\bb x^\alpha).
$\\
In light of \eqref{permFormula}, we can now view forward differences as a kind of ``discrete differential" acting on $\bb f$.
\begin{lemma}
For any $\sigma \in S_{n+1}$ and $\bb f\in C^1(I\times U\rightarrow\mathbb{R}^n)$, \eqref{permFormula} is equivalent to,
\begin{align}
\Delta_{\sigma}\bb f(\bb X^{\bb k}) &= \sum_{i=0}^n \frac{\Delta}{\Delta x_i}\bb f (\bb X^{\bb k+\bb{v}_{\sigma^{-1}(i)}})\Delta x_i \label{discDiffPerm}\\&=\frac{\Delta}{\Delta t}\bb f(\bb X^{\bb k+\bb{v}_{\sigma^{-1}(0)}}) \Delta t+\sum_{i=1}^n \frac{\Delta}{\Delta x_i}\bb f(\bb X^{\bb k+\bb{v}_{\sigma^{-1}(i)}}) \Delta x_{i}. \nonumber
\end{align}Moreover, the symmetrized version follows from \eqref{symForDiv},
\begin{align}
\Delta_{sym} \bb f(\bb X^{\bb k}) &= \sum_{i=0}^n \frac{\Delta}{\Delta x_i}\bb f_{sym}(\bb X^{\bb k})\Delta x_i \label{discDiffSym} \\
&=\frac{\Delta}{\Delta t}\bb f_{sym}(\bb X^{\bb k}) \Delta t+\sum_{i=1}^n \frac{\Delta}{\Delta x_i}\bb f_{sym}(\bb X^{\bb k}) \Delta x_{i}. \nonumber
\end{align}
\end{lemma}
Similar to calculus rules for differentiable functions, we have the following calculus rules\footnote{By no means, this list is complete. We have only included rules employed in the examples.} for first order divided differences. Their proofs follows by direct computation.

\begin{theorem} For $i=0,\dots, n$, let $\frac{\Delta}{\Delta x_i}$ be the $i$-th divided difference and $\bb X^{\bb k},$ $\bb X^{\bb k+\bb e_i} \in I\times U$. The following relations hold:\\

\begin{enumerate}[label=(\Roman*)]
\item {\bf Linearity} \label{rule:linearity} \\ For $\bb f, \bb g \in C^1(I\times U\rightarrow\mathbb{R}^n)$ and $a,b\in \mathbb{R}$,
\[
\frac{\Delta}{\Delta x_i} \left(a \bb f+b \bb g\right)(\bb X^{\bb k}) = a\frac{\Delta}{\Delta x_i} \bb f(\bb X^{\bb k})+b\frac{\Delta}{\Delta x_i}\bb g(\bb X^{\bb k}).
\]

\item {\bf Product rule} \label{rule:product} \\ For $f,g \in C^1(I\times U\rightarrow\mathbb{R})$,

\[
\frac{\Delta}{\Delta x_i} \left(fg\right)(\bb X^{\bb k}) &= \frac{\Delta }{\Delta x_i}f(\bb X^{\bb k})g(\bb X^{\bb k+\bb e_i})+f(\bb X^{\bb k})\frac{\Delta }{\Delta x_i}g(\bb X^{\bb k})\\
&= \frac{\Delta}{\Delta x_i}f(\bb X^{\bb k})g(\bb X^{\bb k})+f(\bb X^{\bb k+\bb e_i})\frac{\Delta}{\Delta x_i}g(\bb X^{\bb k}).
\]

\item {\bf Reciprocal rule} \label{rule:reciprocal} \\ For $f \in C^1(I\times U\rightarrow\mathbb{R})$ and $0 \notin f(I\times U)$,

\[
&\frac{\Delta}{\Delta x_i} \left(\frac{1}{f}\right)(\bb X^{\bb k}) = -\frac{1}{f(\bb X^{\bb k+\bb e_i})f(\bb X^{\bb k})}\frac{\Delta}{\Delta x_i}f(\bb X^{\bb k}).
\]

\item {\bf Quotient rule} \label{rule:quotient} \\ For $f,g \in C^1(I\times U\rightarrow\mathbb{R})$ and $0 \notin g(I\times U)$,

\[
\frac{\Delta}{\Delta x_i} \left(\frac{f}{g}\right) &= \frac{\frac{\Delta}{\Delta x_i}f(\bb X^{\bb k}) g(\bb X^{\bb k+\bb e_i})-f(\bb X^{\bb k})\frac{\Delta}{\Delta x_i}g(\bb X^{\bb k})}{g(\bb X^{\bb k+\bb e_i})g(\bb X^{\bb k})} \\
&= \frac{\frac{\Delta}{\Delta x_i}f(\bb X^{\bb k}) g(\bb X^{\bb k})-f(\bb X^{\bb k+\bb e_i})\frac{\Delta}{\Delta x_i}g(\bb X^{\bb k})}{g(\bb X^{\bb k+\bb e_i})g(\bb X^{\bb k})}.
\]

\item {\bf Chain rule} \label{rule:chainScalar} \\
Let $g\in C^1(I\times U\rightarrow \mathbb{R})$ and let $f\in C^1(V\rightarrow \mathbb{R})$ with $V\subset \mathbb{R}$ be an open subset such that $g(I\times U)\subset V$. If $\Delta_i g(\bb X^{\bb k})\neq 0$, 
\begin{equation*}
\frac{\Delta}{\Delta x_i} (f\circ g)(\bb X^{\bb k}) = \frac{\Delta_i f(g(\bb X^{\bb k}))}{\Delta_i g(\bb X^{\bb k})}\frac{\Delta}{\Delta x_i}g(\bb X^{\bb k}). 
\end{equation*}



\item {\bf Constant rule} \label{rule:constant}\\
Let $\bb f({\bb X}^\alpha)=\bb f(x_0^{\alpha_0},\dots,x_{i-1}^{\alpha_{i-1}},x_{i+1}^{\alpha_{i+1}}\dots,x_n^{\alpha_n})$.
\[
\frac{\Delta }{\Delta x_i}\bb f({\bb X}^\alpha) = \bb 0.
\]

\item {\bf Separable product rule} \label{rule:sepProd}\\
Let $f({\bb X}^\alpha)=\prod_{i=0}^n f_i(x_i^{\alpha_i})$ with single variable functions $f_i\in C^1(I\times U\rightarrow\mathbb{R})$.
\[
\frac{\Delta }{\Delta x_i}f({\bb X}^\alpha) = \frac{\Delta}{\Delta x_i}f_i(x^{\alpha_i}_i) \prod_{i\neq j=0}^n f_j(x^{\alpha_j}_j) .
\]


\item {\bf Rational power rule} \label{rule:rationalPower}\\
Let $p,q\in \mathbb{N}$. If $q>1$, assume $x_i^k, x_i^{k+1}$ are positive.
\begin{equation*}
\frac{\Delta}{\Delta x_i}\left((x^{k}_i)^\frac{p}{q}\right) = 
\frac{\sum_{l=0}^{p-1} (x^{k+1}_i)^{\frac{l}{q}} (x^{k}_i)^{\frac{p-1-l}{q}}}{\sum_{l=0}^{q-1} (x^{k+1}_i)^\frac{l}{q} (x^{k}_i)^\frac{q-1-l}{q}}.
\end{equation*}

\item {\bf Multivariate polynomial rule} \label{rule:polynomial}\\
Let $\displaystyle f(\bb X^{\bb \alpha})=\sum_{|\bb p|\leq d} c_{\bb p} \prod_{i=0}^n (x_i^{\alpha_i})^{p_i}$ be a $(n+1)$-variate polynomial of at most degree $d$ with $\bb p=(p_0,\dots, p_n)\in\mathbb{N}_0^{n+1}$ and $c_{\bb p} \in \mathbb{R}$.
\[
\frac{\Delta}{\Delta x_i} f(\bb X^{\bb \alpha}) = \sum_{|\bb p|\leq d} c_{\bb p}  \left(\sum_{l=0}^{p_i-1} (x^{\alpha_i+1}_i)^l (x^{\alpha_i}_i)^{p_i-1-l}\right) \prod_{i\neq j=0}^n (x_j^{\alpha_j})^{p_j}.
\]

\item {\bf Exponential rule} \label{rule:exponential}\\
\[
\frac{\Delta}{\Delta x_i} e^{x_i^k} = e^{x_i^k} \left(\frac{e^{\Delta x_i}-1}{\Delta x_i}\right).
\]

\item {\bf Logarithm rule} \label{rule:logarithm} \\
\[
\frac{\Delta}{\Delta x_i} \log(x_i^k) = \frac{\log x_i^{k+1}-\log x_i^k}{x_i^{k+1}-x_i^k} = \frac{1}{x_i^k} \left(\frac{\log \left(\frac{x_i^{k+1}}{x_i^k}\right)}{\left(\frac{x_i^{k+1}}{x_i^k}\right)-1}\right).
\]
\end{enumerate}
\end{theorem}


\bibliographystyle{siamplain}

\bibliography{refs}

\end{document}